\documentclass{article}

\usepackage{quickmath}
\usepackage{ifthen}
\usepackage{varwidth}
\usepackage[modulo]{lineno}

\newcommand{\ti}{\tilde}
\renewcommand{\ddt}{\frac{\dd}{\dd t}}
\newcommand{\MRprop}{\mathcal{M{\kern-0.5pt}R}}
\newcommand{\mrprop}{m{\kern-0.5pt}r}

\DeclareMathOperator{\MR}{M{\kern-0.5pt}R}

\usepackage{authblk} 

\title{Quantitative estimates of $L^p$ maximal regularity for nonautonomous operators and global existence for quasilinear equations}
\author[1,2,3]{Théo Belin}
\author[1,3]{Pauline Lafitte}
\affil[1]{Université Paris-Saclay, Fédération de Mathématiques de CentraleSupélec (FR CNRS 3487)}
\affil[2]{Université Paris-Saclay, Laboratoire MICS, CentraleSupélec}
\affil[3]{Université Paris-Saclay}

\date{}

\begin{document} 

\maketitle
\tableofcontents

\newpage 

{\centering\section*{Abstract}}
In this work, we obtain quantitative estimates of the continuity constant for the $L^p$ maximal regularity of relatively continuous nonautonomous operators $A : I \ra \mcl{L}(D,X)$, where $D \cinj X$ densely and compactly. They allow in particular to establish a new general growth condition for the global existence of strong solutions of Cauchy problems for  nonlocal quasilinear equations for a certain class of nonlinearities $u \mapsto \mathbb{A}(u)$. The estimates obtained rely on the precise asymptotic analysis of the continuity constant with respect to perturbations of the operator of the form $A(\cdot) + \lambda I$ as $\lambda \ra \pm \infty$. A complementary work in preparation supplements this abstract inquiry with an application of these results to nonlocal parabolic equations in
noncylindrical domains depending on the time variable. 

Keywords: Nonautonomous Cauchy problem, $L^p$ maximal regularity, relatively continuous operators, quasilinear equations.

\section{Introduction}
The $L^p$ maximal regularity theory deals with the well-posedness in the strong sense of the following abstract nonautonomous Cauchy problem,
\begin{equation}
\label{eq:nacp_intro}
\left\{
\begin{array}{rclr}
\displaystyle \ddt u(t) + A(t)u(t) &=& f(t) & t\in I = (a,b)\\
u(a) &=& 0
\end{array}
\right.
\end{equation}
The problem is set on a time interval $I = (a,b)\subset \R$ bounded from below, the source term $f$ lies in $L^p(I;X)$ and $A: I \lra \mcl{L}(D(A);X)$ is a nonautonomous linear operator with a uniform-in-time domain $D(A) \subset X$. \\

For a generic source term $f \in L^p(I;X)$, a strong solution of \eqref{eq:nacp_intro}, i.e. a solution $u$ such that $t \mapsto \sbk*{\ddt u}(t), A(t)u(t)$ are elements of $L^p(I;X)$, lies in the Banach space $\MR^p(I):= W^{1,p}(I;X)\cap L^p(I;D(A))$ called the \emph{maximal regularity space}. Given the regularity of the source term ${f \in L^p(I;X)}$, this is the best regularity one can possibly obtain for $u$, giving a precise semantic meaning to the terminology of $L^p$ \emph{maximal regularity}. More precisely if there exists $K \geq 0$ such that for any $f\in L^p(I;X)$ there exists a unique solution $u \in \MR^p(I)$ to \eqref{eq:nacp_intro} and
\begin{equation}
\label{eq:intro_MR_const}
\norm*{u}_{\MR^p(I)} \leq K \norm*{f}_{L^p(I;X)}
\end{equation}
then we say that $A$ is \emph{maximally regular} in $L^p$. \\

It is of interest to study the following questions which have numerous applications, {\em e.g.} to the study of the well-posedness and regularity of solutions of Cauchy problems for general parabolic equations:\\

\begin{enumerate}[label = (Q\arabic*)]
\item \label{item:q1} Under which conditions on $A$ is \eqref{eq:nacp_intro} well-posed with respect to the source term $f \in L^p(I; X)$?
\item \label{item:q2} Can we identify quantities related to the operator $A$ which give fine estimates and precise asymptotic behaviors of the constant $K$ in \eqref{eq:intro_MR_const}?
\item \label{item:q3} In which topologies on the set of nonautonomous operators is the set of maximally regular operators closed? In these topologies, does the continuity of strong solutions in $\MR^p(I)$ hold?
\item \label{item:q4} Can answering \ref{item:q2} and \ref{item:q3}
  allow to infer well-posedness criteria of the Cauchy problem for nonlocal quasilinear equations of the form ${\ddt u(t) + \mathbb{A}(u, t) u = \mathbb{F}(u, t)}$, $u(0) = x$?
\end{enumerate}


Logically, answers to questions \ref{item:q2},  \ref{item:q3}
and \ref{item:q4} must rely on the already existing answers to question \ref{item:q1}, and we hereby give a short summary of the standard results found for $L^p$ maximal regularity regarding \ref{item:q1} for autonomous and nonautonomous operators. \\

For \emph{autonomous operators} $A(t) = A$, \ref{item:q1} has received a lot of attention. Since de Simon \cite{de-simon1964unapplicazione} it is known that
if $X$ is a Hilbert space, the maximal regularity on the unbounded interval $(0,+\infty)$ is equivalent to the sectoriality of the operator $-A$ which is also equivalent to the analyticity of the semi-group generated by $-A$ on a sector of the complex plane. In the setting where $X$ is merely a Banach space, the sectoriality of $-A$ is necessary (see e.g. \cite{monniaux2009maximal}) but not sufficient in general. \\
More recently, in \cite{kunstmann2004maximal}, it has been proved that the $\mcl{R}$-boundedness of the resolvent operator of $A$ yields maximal regularity in Banach spaces $X$ of the $UMD$ class, that is, the Banach spaces for which the Hilbert transform is a bounded operator in $L^q(\R;X)$, for some $1 < q < +\infty$. \\
For a thorough introduction on maximal regularity for autonomous operators we refer the reader to the lecture notes of Kunstmann and Weis \cite{kunstmann2004maximal}.\\

The study of the autonomous case relies on a Fourier transform in the time variable and the study of the regularizing properties of the subsequent kernel. Such a technique logically leads to difficulties in the \emph{nonautonomous} case. Sufficient conditions known in the literature for maximal regularity usually require a certain degree of regularity in time of $A(\cdot)$ as well as the maximal regularity of each $A(t)$, $t\in \ol{I}$.\\
The first notable results were obtained by Acquistapace and Terreni \cite{acquistapace1987a-unified} who have uncovered a Hölder-type condition in time of the operator, which may have a time-dependent domain $D(A(t))$. In the case of a constant domain $D(A)$, which is our interest here, Prüss and Schnaubelt \cite{pruss2001solvability} showed that the strong continuity in time along with the maximal regularity of each $A(t)$ is sufficient to guarantee the maximal regularity of the nonautonomous operator. These inquiries were further refined and developed by Amann \cite{amann2004maximal}, Yagi \cite[Chapter 3]{yagi2009abstract}, Gallarati and Veraar \cite{gallarati2017maximal}. In this direction, Arendt and coauthors \cite{arendt2007lp-maximal} gave a weaker sufficient time-regularity condition called \emph{relative continuity}. \\

Overall, on the subject of nonautonomous maximal regularity, we refer the reader to the reviews by Monniaux \cite{monniaux2009maximal} or most recently by Pyatkov \cite{pyatkov2019on-the-maximal}.\\

We rely on relative continuity assumptions to provide new answers to questions \ref{item:q2}, \ref{item:q3}
and \ref{item:q4}. We introduce the space $RC(I; (D,X))$ of relatively continuous operators and isolate meaningful quantities through which one can estimate the constant $K$. \\

The $L^p$ maximal regularity has shown its wide uses in various settings for the study of strong solutions of very wide classes of parabolic-type equations. Again we refer the reader to the monograph \cite{monniaux2009maximal} and references therein. The framework of maximal regularity is a crucial tool to study the strong solutions of nonlinear equations. Let us mention the recent work of Danchin and coauthors \cite{danchin2023free} who explore the critical setting $p = 1$. It turns out that this $L^1$ setting is natural to study the equivalence of Lagrangian and Eulerian formulations globally in time. Their inquiries lead to well-posedness results in free boundary problems of fluid mechanics.

Answers to \ref{item:q4} can rely on answers to \ref{item:q2}. To fix some ideas, we focus our interest on a class of \emph{nonlocal} quasilinear equations of the form 

\begin{equation}
\label{eq:non_local_quasilinear_intro}
\left\{
\begin{array}{rclr}
\displaystyle \ddt u(t) + \mathbb{A}(u,t)u(t) &=& \mathbb{F}(u,t) & t\in I\\
u(a) &=& x
\end{array}
\right.
\end{equation}
on a bounded time interval $I = (a,b) \subset \R$, $b < +\infty$. The nonlinearities $u \mapsto \mathbb{A}(u)$ and $u \mapsto \mathbb{F}(u)$ are assumed to be defined on $\mcl{X}(I)$, a functional space on $I$ for which the following chain of embeddings holds: 
\begin{equation*}
\MR^p(I) \cinj_c \mcl{X}(I) \cinj L^p(I;X)
\end{equation*}
where the first embedding is compact. As such these nonlinear operators may include nonlocal effects in time. They are sometimes assumed to satisfy a certain \emph{causality} principle translated by the Volterra property (see Amann \cite{amann2005quasilinear}): the restriction of $\mathbb{A}(u)$ and $\mathbb{F}(u)$ to any subinterval of the form $(a,t)$, $t \leq b$ only depends on $u|_{I_t}$. \\

These equations appear in various models of elasticity and solid mechanics with memory and nonlocal effects (see \cite{clement1992global, amann2006quasilinear}, \cite{du2017analysis} and references therein) or population dynamics \cite{kosovalic2013abstract}. They generally model phenomena which have a memory of their past states. Their interest is both found in applications and in theoretical studies since they also appear when regularizing certain singular equations, for instance in the Perona-Malik model of anisotropic diffusion \cite{amann2007time, guidotti2009new, guidotti2015anisotropic}. \\

The well-posedness of these Cauchy problems in an abstract framework has been studied by Amann (see \cite{amann1995linear, amann2004maximal, amann2005quasilinear, amann2006quasilinear}). His most important result in \cite{amann2005quasilinear} shows that if $\mcl{X}(I) = \MR^p(I)$ and under the Lipschitz continuity of $u \mapsto \mathbb{A}(u)$ in $L^\infty(I;\mcl{L}(D,X))$ and $u \mapsto \mathbb{F}(u)$ in $L^p(I;X)$ then the local well-posedness of \eqref{eq:non_local_quasilinear_intro} holds. Amann further shows existence of a Lipschitz semi-flow associated to the equation. These results rely on a Banach fixed point theorem which inherently requires the Lipschitz continuity of the operators. \\

Since the embedding $D \cinj X$ is compact, we rely here on Schauder's fixed point theorem to obtain existence results. No Lipschitz continuity is required, thereby permitting lower regularity assumptions on the nonlinearities $\mathbb{A}$ and a mere weak continuity condition for $\mathbb{F}$ (see e.g. Arendt \cite{arendt2010global}). The existence is guaranteed through growth conditions on the regularity constant $K_{\mathbb{A}(u)}$ in the regime $\norm*{u}_{\mcl{X}(I)} \ra +\infty$. We use our answers to \ref{item:q2} in the setting of relatively continuous operators of Hölder-type to obtain quantitative growth conditions on $\mathbb{A}(u)$ and $\mathbb{F}(u)$. An interesting interplay between the pointwise maximal regularity constant of each operator $\mathbb{A}(u)$ and their Hölder semi-norm is observed. 

In \Cref{section:main_results}, we introduce the notion of $L^p$ maximal regularity on bounded and unbounded intervals as well as the maximal regularity constant. We describe the maximal regularity space $\MR^p(I)$ and a special discussion is dedicated to scale-invariant norms. Then relatively continuous operators and their \emph{ranges of relative continuity} are introduced. Related preliminary results pertaining to relatively continuous operators are further presented. This context allows the exposition of our main result: we first present the quantitative behavior of the regularity constant with respect to perturbations of the form $A(\cdot) + \lambda I$, in particular in the asymptotic regime $\lambda \ra \pm \infty$ (\autoref{thm:perturbation_dilation}), we then show our main quantitative estimates of the regularity constant for relatively continuous operators (\autoref{thm:relative_continuity}, \autoref{thm:A_MR_log_estimate}). \\
As an application to the estimate in \autoref{thm:relative_continuity}, we present various conditions giving positive answers to \ref{item:q3} (\autoref{cor:uniform_mr_bound},\autoref{cor:weak-weak_stability}, \autoref{cor:strong-strong_stability}).
The growth condition for the global existence for quasilinear equations is then presented for a Hölder-type class of relatively continuous operators \autoref{thm:existence_abstract}, \autoref{thm:existence_explicit}.\\

In \Cref{section:2} we describe the operators $L^I_A:= \ddt + \mcl{A}^I$ related to the Cauchy problem on $I \subset \R$ and recall known properties in the autonomous case. We then prove precise perturbation results with asymptotic behavior through fine manipulations of the chosen norm of the maximal regularity space. \\
\Cref{section:3} is dedicated to the proof of intermediate results, in particular \autoref{thm:finite_glueing} which provides the frame of the quantitative estimate \autoref{thm:relative_continuity} proved in \Cref{section:4}. \\
Lastly \Cref{section:5} treats the global existence of nonlocal quasilinear equations for a class of relatively continuous operators. Some proofs, which are classical or standard, were omitted here, but can be found in \cite{belin2024wp}. 

\section{Main results}
\label{section:main_results}

In the quantitative study of $L^p$ maximal regularity, the choice of the norm in the maximal regularity space $\MR^p(I)$ can have a noticeable impact and should be chosen with care. Hence in \Cref{subsection:operators} we set basic notations for operators and norms and we motivate our choices for the norm of the maximal regularity space via time-scaling invariance. We further introduce in \Cref{subsection:maximal} basic notations of $L^p$ maximal regularity and elementary results. \Cref{subsection:relatively} contains an introduction to the class of relatively continuous operators and we describe topological properties of this space of operators. \Cref{subsection:main} finally showcases our main results.

\subsection{Operators, spaces and measurability}
\label{subsection:operators}

\paragraph{Graph norms and domains of operators}

If two norms $N_1$ and $N_2$ are \emph{equivalent} in a given Banach space, meaning there exist $c, C > 0$ such that

\begin{equation*}
cN_2 \leq N_1 \leq CN_2
\end{equation*}
we write $N_1 \overset{c, C}{\sim} N_2$ or equivalently $N_2 \overset{\frac{1}{C}, \frac{1}{c}}{\sim} N_1$. If the constants are universal or not of interest for further development we may simply write $N_1 \sim N_2$.\\

A linear operator $A: D(A) \ra X$ on a Banach space $X$, defined on a linear subspace $D(A) \subset X$, is said to be \emph{closed} if its graph $\Gamma(A):= \cbk*{(x, Ax) ; x \in D(A)}$ is closed. In such a case, we can equip $D(A)$ with the $p$-\emph{graph norm} induced by $A$, $p\in [1,+\infty)$, and defined by 
\begin{equation*}
\norm*{\cdot}_{A, p}:= \pr*{\norm*{\cdot}^p_X + \norm*{A\cdot}^p_X}^{\frac{1}{p}}.
\end{equation*}
The domain $D(A)$ is then a Banach space. It is straightforward to see that for any two ${p,q \in (1,+\infty)}$,
$\norm*{\cdot}_{A,p} \overset{\frac{1}{2}, 2}{\sim} \norm*{\cdot}_{A,q}$. Since the constants of equivalence are universal, we denote without ambiguity $\norm*{.}_A$ for any $p$-graph norm of $A$. \\ 


%
%
%

Let $(X, \norm*{\cdot}_X)$ be a Banach space and let $D\oset{d}{\cinj} X$ be a Banach space with norm denoted $\norm*{\cdot}_D$, which injects densely and compactly in $X$. We denote by $\mcl{L}(D; X)$ the space of \emph{bounded linear operators} from $D$ to $X$. For the rest of our developments $X$ and $D$ are fixed, as well as $p,q \in (1,+\infty)$ two conjugate exponents satisfying $\frac{1}{p} + \frac{1}{q} = 1$. 

\paragraph{Functional spaces for $L^p$ maximal regularity}

Let $I \subseteq \R$ be any open interval bounded from below i.e. $I = (a,b)$, with $-\infty< a < b \leq +\infty$ and $Y$ be a Banach space. \\

We define the vector valued Lebesgue space $L^p(I;Y)$ as the space of Bochner-measurable functions $u: I \lra Y$ such that
\begin{equation*}
\norm*{u}_{L^p(I;Y)} = \pr*{\int_I \norm*{u(t)}_Y^p\dd t}^\frac{1}{p} < +\infty
\end{equation*}

We denote by $W^{1,p}(I;Y)$ the $Y$-valued Sobolev space. It is the space of functions $u \in L^p(I;Y)$ which admit a weak derivative $\ddt u \in L^p(I;Y)$ which satisfies 
\begin{equation*}
- \int_I \sbk*{\ddt \phi}(t) u(t) \dd t = \int_I \phi(t) \sbk*{\ddt u}(t) \dd t
\end{equation*}
for all $\phi \in C^\infty_c(I;\R)$. It is endowed with the semi-norm

\begin{equation}
\label{eq:w1p_seminorm}
\sbk*{u}_{W^{1,p}(I;Y)}:= \norm*{\ddt u}_{L^p(I;Y)} = \pr*{\int_I \norm*{\sbk*{\ddt u}(t)}^p_Y\dd t}^{\frac{1}{p}}
\end{equation}

As usual in the spaces $L^p(I;Y)$ and $W^{1,p}(I;Y)$ we identify two elements $u$ and $v$ if they agree $\mcl{L}^1$-almost everywhere on $I$ for the Lebesgue measure $\mcl{L}^1$. Throughout we endow the equality relation $=$ with this meaning when we identify elements of such spaces.\\

Hereafter we define the maximal regularity space on $I$ along with its norm. 

\begin{df}
\label{df:mrp_space}
Define $\MR^p(I)$, the $L^p$ \emph{maximal regularity space} as
\begin{equation*}
\MR^p(I):= L^p(I;D) \cap W^{1,p}(I;X)
\end{equation*}
If $b < +\infty$ we choose the following norm
\begin{equation}
\label{eq:def_mr_local_norm}
\norm*{u}_{\MR^p(I)}:= \pr*{\norm*{u}^p_{L^p(I;X)} + \norm*{u}^p_{L^p(I;D)} + |I|^p\sbk*{u}_{W^{1,p}(I;X)} }^{\frac{1}{p}}
\end{equation}
Otherwise if $b = +\infty$ then
\begin{equation}
\label{eq:def_mr_usual_norm}
\norm*{u}_{\MR^p(I)}:= \pr*{\norm*{u}^p_{L^p(I;X)} + \norm*{u}^p_{L^p(I;D)} + \sbk*{u}_{W^{1,p}(I;X)}}^{\frac{1}{p}}
\end{equation}
\end{df}

The norm \eqref{eq:def_mr_usual_norm} is more classical: it allows for a unified treatment of maximal regularity on bounded and unbounded intervals. However in our work, we mostly study maximal regularity in the bounded setting $b < +\infty$. The norm \eqref{eq:def_mr_local_norm} is well behaved under time rescalings thanks to the weight $|I|^p$ in front of the $W^{1,p}$ semi-norm. In fact, for any $\lambda \neq 0$ a natural \emph{isometry} from $\MR^p(I)$ to $\MR^p(\lambda I)$ can be defined as $\iota_\lambda: u \mapsto \lambda^{-\frac{1}{p}}u(\lambda^{-1}\cdot)$. This rescaling property will prove useful in the development of the arguments throughout. \\

For further use, let us also define the closed subspace $\MR^p_0(I) \subset \MR^p(I)$ defined as follows

\begin{equation}
\label{eq:def_mr_0}
\MR^p_0(I):= \cbk*{u \in \MR^p(I) \:\  u(a) = 0} 
\end{equation}

\begin{rem}
\label{rem:trace_space}
It makes sense to evaluate $u \in \MR^p(I)$ at any point of $I$ (e.g. $u(a)$) as the following continuous embedding ${\MR^p(I) \hookrightarrow C\pr*{\ol{I};\Tr^p}}$ holds, where $\Tr^p:= (D, X)_{\frac{1}{q}, p}$ is a \emph{real interpolation} between $D$ and $X$ of parameters $\pr*{\frac{1}{q},p}$ (see e.g. \cite[Theorem 3.12.2]{lofstrom1976interpolation}). Equivalently $\Tr^p$ can be interpreted as a \emph{trace space} of $\MR^p(I)$ meaning ${\Tr^p = \cbk*{u(a)\:\ u \in \MR^p(I)}}$ equipped with the norm 

\begin{equation*}
\norm*{x}_{\Tr^p, I} = \inf\cbk*{\norm*{u}_{\MR^p(I)}\: \ u \in \MR^p(I),\ u(a) = x} 
\end{equation*}

In the case of a bounded interval $b < +\infty$, under the isometry $\iota_\lambda:\MR^p(I) \ra \MR^p(\lambda I)$ the norm of the trace space is scaled accordingly. More precisely we have for any $x \in \Tr^p$,

\begin{align*}
\norm*{x}_{\Tr^p, I} &= \inf \cbk*{\norm*{u}_{\MR^p(I)}\: \ u \in \MR^p(I),\ u(a) = x} \\
&= \inf \cbk*{\norm*{\iota_\lambda^{-1}(v)}_{\MR^p(I)}\: \ v \in \MR^p(\lambda I),\ \iota_\lambda^{-1}(v)(a) = x} \\
&= \inf \cbk*{\norm*{v}_{\MR^p(\lambda I)}\: \ v \in \MR^p(\lambda I),\ \lambda^{\frac{1}{p}}v(\lambda a) = x}  \\
&= |\lambda|^{-\frac{1}{p}}\norm*{x}_{\Tr^p,\lambda I}
\end{align*}

Using the time-reversing rescaling $\iota_{-1}$ we see that the trace can either be taken on the left or on the right of the interval in the definition i.e. $\Tr^p = \cbk*{u(b) \:\ u \in \MR^p(a,b)}$. 
\end{rem}

\paragraph{Strong measurability of nonautonomous operators}

Let $I = (a,b)\subset \R$ be an open interval, bounded from below. For nonautonomous operators $A:\ol{I} \mapsto \mcl{L}(D;X)$, the notion of \emph{strong measurability} is necessary to allow the definition of the associated multiplication operators $\mcl{A}$ defined in the next paragraph. We want to give meaning to the quantity

\begin{equation*}
\norm*{\mcl{A}v}_{L^p(I;X)} = \pr*{\int_I \norm*{A(t)v(t)}^p_X\dd t}^{\frac{1}{p}}
\end{equation*}
which requires the measurability of the map $t \mapsto \norm*{A(t)v(t)}_X$. And if $t \mapsto A(t)$ is \emph{strongly measurable} as defined below, the measurability of this map will hold for any $v \in L^p(I;D)$.

\begin{df}
A nonautonomous operator $A: \ol{I} \rightarrow \mcl{L}(D; X)$ is said to be \emph{strongly measurable} if for any $x \in D$, the map $A(\cdot)x: I \rightarrow X$ is \emph{Bochner-measurable} in $X$. 
\end{df}

The following remarks motivate the use of this notion of strong measurability for applications.  

\begin{rem}
\begin{enumerate}
\item The notion of strong measurability for operators is a weaker type of measurability than the measurability of $t \mapsto A(t)$ in $\mcl{L}(D;X)$. Even with $D$ and $X$ separable, these two notions of measurability are not equivalent in general.

\item By virtue of Pettis' theorem \cite{pettis1938on-integration}, since $X$ is separable, if $t\mapsto A(t)x$ is weakly continuous in $X$ for any $x \in D$, then $A$ is strongly measurable. Pettis' theorem allows to show the strong measurability of differential operators with continuous coefficients. For example fix $c \in C_b(\ol{I}\times\R)$, let $X = L^2(\R)$, $D = H^1(\R)$ and define for each $t \in \ol{I}$ the operator $A(t)u:= \cbk*{x \mapsto c(t,x) \frac{\dd}{\dd x} u(x)} \in \mcl{L}(D, X)$. We see that $t \mapsto \sbk*{A(t)}u = c(t, x)\frac{\dd}{\dd x} u(x)$ is weakly continuous in $L^2(\R)$ if $u \in H^1(\R)$. Since $L^2(\R)$ is separable, Pettis' theorem applies and we obtain the strong measurability of $A$.
\end{enumerate}

\end{rem}

We may then equip this notion of strong measurability to Lebesgue spaces of nonautonomous operators. 

\begin{df}
\label{df:operator_lebesgue_space}
We define the Lebesgue space 

\begin{equation*}
L^\infty\pr*{I;\mcl{L}(D;X)}:= 
\cbk*{
A: I \ra \mcl{L}(D;X) \ \left| \  
\begin{array}{l}
A \text{ is strongly measurable}\\
\\
\qquad \text{ and }\\
\\
\text{ For any } x \in D, \norm*{A(\cdot) x}_X \in L^\infty(I;\R) 
\end{array}
\right.
}
\end{equation*}
\end{df}

The separability of $D$ and the uniform boundedness principle ensure that the map $t \mapsto \norm*{A(t)}_{\mcl{L}(D;X)}$ is essentially bounded in $I$ for any $A \in L^\infty\pr*{I;\mcl{L}(D;X)}$.\\
As usual, the identification in $L^\infty\pr*{I;\mcl{L}(D;X)}$ is performed almost everywhere in $I$ (for the one-dimensional Lebesgue measure) and the norm $\norm*{A}_{L^\infty(I;\mcl{L}(D;X))}:= \norm*{\norm*{A(\cdot)}_{\mcl{L}(D;X)}}_{L^\infty(I)}$ gives it the structure of a Banach space. When there is no ambiguity, we shall often write $\norm*{A}_{\infty}$ for this norm.

\subsection{$L^p$ maximal regularity}
\label{subsection:maximal}

Let $I = (a,b) \subseteq \R$ be an open interval bounded from below, i.e. $-\infty < a < b \leq +\infty$. We define the maximal regularity property and the maximal regularity constant for autonomous and nonautonomous operators on $I$. We then derive elementary properties of the classes defined below, in particular the stability of the maximal regularity property under restriction of the time interval.

\paragraph{Autonomous operators: $\mrprop_p(I)$}

\begin{df}
\label{df:mrp}
An operator $A \in \mcl{L}(D; X)$ has the $L^p$ \emph{maximal regularity property} on $I$ if 
\begin{enumerate}[label=\emph{(aut.\roman*)}]
\item \label{df:Ai} $\norm*{\cdot}_A \sim \norm*{\cdot}_D$.
\item \label{df:Aii} For each $f \in L^p(I;X)$ there exists a unique solution $v^A_f \in \MR^p(I)$ of the following abstract Cauchy problem,
\begin{equation}
\label{eq:acp_0}
\left\{
\begin{array}{rclr}
\displaystyle \ddt v(t) + Av(t) &=& f(t) & \text{ for } t \in I\\
v(a) &=& 0
\end{array}
\right.
\end{equation}
and there exists $K \geq 0$, independent of $f$ such that 
\begin{equation}
\label{eq:mr_ineq}
\norm*{v^A_f}_{\MR^p(I)} \leq K \norm*{f}_{L^p(I; X)}.
\end{equation}
\end{enumerate}

We denote $\mrprop_p(I)$ this family of maximally regular autonomous operators and for each $A \in \mrprop_p(I)$, the infimum $K$ in \eqref{eq:mr_ineq} is denoted by $\sbk*{A}_{\mrprop_p(I)}$. It is called the \emph{maximal regularity constant} of $A$ on $I$. 
\end{df}

\paragraph{Nonautonomous operators: $\MRprop_p(I)$}

\begin{df}
\label{df:MRp}
A nonautonomous operator $A\in L^\infty(I;\mcl{L}(D;X))$ has the $L^p$ \emph{maximal regularity property} on $I$ if 
\begin{enumerate}[label=\emph{(nonaut-\roman*)}]
\item \label{df:NAi} For any $t \in I$, $\norm*{\cdot}_{A(t)} \sim \norm*{\cdot}_D$.
\item \label{df:NAii} For any $f \in L^p(I;X)$ there exists a unique solution $v^A_f \in \MR^p(I)$ of the following abstract \emph{nonautonomous} Cauchy problem

\begin{equation}
\label{eq:nacp_0}
\left\{
\begin{array}{rclr}
\displaystyle \ddt v(t) + A(t)v(t) &=& f(t) & t \in I\\
v(a) &=& 0
\end{array}
\right.
\end{equation}

And there exists $K\geq 0$, independent of $f$ such that 
\begin{equation}
\label{eq:MR_ineq}
\norm*{v^A_f}_{\MR^p(I)} \leq K \norm*{f}_{L^p(I; X)}
\end{equation}
\end{enumerate}

We denote $\MRprop_p(I)$ this family of nonautonomous and maximally regular operators and for each $A \in \MRprop_p(I)$ the infimum $K$ in \eqref{eq:MR_ineq} is denoted by $\sbk*{A}_{\MRprop_p(I)}$. It is \emph{also} called the \emph{maximal regularity constant} of $A$.
\end{df}

Following these definitions several clarifying remarks are in order. 

\begin{rem}
\begin{enumerate}[label = $(\roman*)$]
\item \label{rem:i} The condition $\norm*{\cdot}_A \sim \norm*{\cdot}_D$ is equivalent to asking that $A$ be a \emph{closed unbounded} operator on $X$ with domain $D$, provided $D \subsetneq X$. If $D = X$, then maximal regularity will hold for all $A \in \mcl{L}(X)$ by virtue of the abstract Cauchy-Lipschitz theorem. In the case $X = L^p(\Omega)$, $D = W^{2,p}(\Omega)\cap W^{1,p}_0(\Omega)$ it is known that autonomous second order and uniformly elliptic operators with continuous coefficients are maximally regular, condition \ref{df:Ai} being given by the theory of elliptic regularity (see e.g. \cite{gilbarg1977elliptic}). 

\item \label{rem:ii} See that $\mrprop_p(I)$ (resp. $\MRprop_p(I)$) is not stable by finite sum because requirement \ref{df:Ai} (resp. \ref{df:NAi}) fails in general for the sum of two operators. \\


\item \label{rem:iii} When the time interval $I$ is clear by context, we shall write $\mrprop_p$ and $\MRprop_p$ instead of $\mrprop_p(I)$ or $\MRprop_p(I)$ respectively. Moreover we easily see that $\mrprop_p(I) \subset \MRprop_p(I)$. 

\item \label{rem:iv} For autonomous operators, $\mrprop_p(I) = \mrprop_p(J)$ for any two bounded intervals $I$, $J \subset \R$, in other words, in the case of bounded intervals, maximal $L^p$ regularity is independent of the interval. However we remark that, according to \autoref{df:mrp} $\mrprop_p(0,1) \neq \mrprop_p(\R_+)$ in general. Fix an open bounded space interval $\Omega \subset \R$ and consider the diagonally perturbed Dirichlet Laplacian operator $A:= -\Delta_D - \lambda\Id$ defined on $D = H^2(\Omega)\cap H^1_0(\Omega) \subset L^2(\Omega)$, where $\lambda \geq 0$ is some eigenvalue of $-\Delta_D$. We have $A \in \mrprop_p(0,1)$ while $A \not\in \mrprop_p(\R_+)$. Indeed, if $u_\lambda \in C^\infty(\ol{\Omega})$ is an eigenvector of $-\Delta_D$ associated to $\lambda$, then $t \in \R_+ \mapsto u_\lambda$, is the unique solution of the homogeneous Cauchy problem with initial condition $u(0) = u_\lambda \in \Tr^p$. Hence $u \in \MR^p(0,1)$ while $u \not \in \MR^p(\R^*_+)$. \\
Another way to see that $A \not \in \mrprop_p(\R^*_+)$ is to remark that it is not a sectorial operator. 
\end{enumerate}
\end{rem}

If we are given a nonautonomous operator $A: \ol{I} \mapsto \mcl{L}(D,X)$, a source term $f \in L^p(I;X)$ and an initial condition $x \in \Tr^p$ we can consider the following Cauchy problem: 

\begin{equation}
\label{eq:cp}
\left\{
\begin{array}{rcll}
\displaystyle \ddt v(t) + A(t) v(t) &= &f(t)& t \in I\\
v(a) &=& x&
\end{array}
\right.
\tag{CP}
\end{equation}
where recall that $I = (a,b) \subset \R$ is bounded from below. We use the notation \eqref{eq:cp}$^I_{x,f}$ to refer to the above Cauchy problem: we explicitly specify the initial datum $x$, the source term $f$ and the time interval $I$ under consideration. It turns out that condition \ref{df:NAii} is equivalent to the well-posedness of the general Cauchy problem \eqref{eq:cp}$^I_{x,f}$ for any $x \in \Tr^p$, $f \in L^p(I;X)$. We omit the proof of this result which relies on standard lifting techniques. 

\begin{prop}
\label{prop:mr_nacp}
Let $A \in L^\infty(I;\mcl{L}(D;X))$ be a nonautonomous operators satisfying \ref{df:NAi}. \\

Then $A \in \MRprop_p(I)$ if and only if for any $x \in \Tr^p$, $f \in L^p(I;X)$ there exists a unique solution to \eqref{eq:cp}$^I_{x,f}$ and there exists $K \geq 0$, independent of $f$ and $x$ such that 

\begin{equation*}
\norm*{v}_{\MR^p(I)} \leq K\pr*{\norm*{f}_{L^p(I;X)} + \norm*{x}_{\Tr^p,I}}
\end{equation*}
\end{prop}

It is quite natural to want to be able to restrict an operator $A \in \MRprop_p(I)$ to any subinterval $J \subset I$ while keeping the maximal regularity property of the restricted operator $A|_J$. This property amounts to the existence of a regular evolution operator for the equation. In our case, the compactness of the embedding $D \cinj X$ is a key ingredient that allows to derive this property. The proof of the result can be found in \cite[Appendix D]{belin2024wp}. 

\begin{prop}
\label{prop:mr_injections}
Let $I \subset \R$ be an open interval bounded from below and let $A \in \MRprop_p(I)$. \\
Then for any subinterval $J \subset I$, $A|_J \in \MRprop_p(J)$. Moreover, if either $|J| = |I| = \infty$ or $|J| \leq |I| < +\infty$ then
\begin{equation}
\label{eq:mr_const_IJ_same_type}
\sbk*{A|_J}_{\MRprop_p(J)} \leq \sbk*{A}_{\MRprop_p(I)}
\end{equation}
Otherwise if $|J| < |I| = +\infty$, then
\begin{equation}
\label{eq:mr_const_IJ_diff_type}
\sbk*{A|_J}_{\MRprop_p(J)} \leq \pr*{|J| \vee 1}\sbk*{A}_{\MRprop_p(I)}
\end{equation}

\end{prop}

\subsection{Relatively continuous operators}
\label{subsection:relatively}

Regarding nonautonomous operators, relative continuity is a weaker type of continuity than strong continuity. Relative continuity incorporates the idea of perturbing continuous operators and as such allows to treat with generality various kinds of such perturbations. The definition, which we reformulate here, was first given by Arendt and coauthors in \cite[Definition 2.5]{arendt2007lp-maximal}. We state elementary topological properties of $RC(I)$, the space of relatively continuous operators on $I$. In particular $RC(I)$ is a closed linear subspace of $L^\infty(I;\mcl{L}(D;X))$. Then we exhibit a theorem, analogous in shape, to the Arzelà-Ascoli theorem for continuous maps. 

\begin{df}
Let $I = (a,b)$ be a bounded open interval.
\begin{itemize}
\item We say that a strongly measurable nonautonomous operator $A: \ol{I} \rightarrow \mcl{L}(D; X)$ is \emph{relatively continuous} if for any $\epsilon > 0$, there exists $\delta > 0$ and $\eta \geq 0$ such that 
\begin{equation}
\tag{$P^\epsilon_{\eta, \delta}$}
\label{eq:def_relative_continuity}
\forall x \in D, \forall t,s \in \ol{I}, \qquad |t-s|\leq \delta \ \Longrightarrow  \ \norm*{A(t)x - A(s)x}_X \leq \epsilon\norm*{x}_D + \eta\norm*{x}_X
\end{equation}

We denote by $RC(I;(D;X))$ the class of nonautonomous operators of $\mcl{L}(D;X)$ which are relatively continuous.

\item For $A \in RC(I;(D;X))$, we call the \emph{ranges of relative continuity} the following set-valued function, which is nondecreasing with respect to set inclusion
\begin{equation*}
r^*_A: 
\left\{
\begin{array}{ccc}
(0,+\infty) &\lra &\mcl{P}\pr{(0,+\infty)\times[0,+\infty)}\\
\epsilon &\longmapsto &\cbk*{(\delta, \eta) \:\ \eqref{eq:def_relative_continuity} \text{ holds }}
\end{array}
\right.
\end{equation*}

\item We say that $r_A = (\delta, \eta): (0,+\infty) \ra (0,+\infty)\times[0,+\infty)$ is a \emph{specific} range of relative continuity of $A$ if
\begin{enumerate}[label = ($\star$), leftmargin = 5em]
\item $r_A(\epsilon) \in r^*_A(\epsilon)$ for any $\epsilon \in (0,+\infty)$.
\item $\epsilon \mapsto \delta(\epsilon)$ is nondecreasing.
\item $\epsilon \mapsto \eta(\epsilon)$ is nonincreasing.
\end{enumerate}
\end{itemize}
\end{df}

Here $D$ and $X$ are fixed, therefore we write $RC(I)$ unambiguously for $RC(I;(D;X))$. Furthermore, when we wish to select a specific range of relative continuity $r_A$, with a slight abuse of notation, we may simply write $r_A \in r^*_A$. If $A \in RC(I)$, then it is always possible to select a specific range of relative continuity for $A$.  \\

We now state some topological facts about relatively continuous operators which are proved in \cite[Appendix B]{belin2024wp}. 

\begin{prop}
\label{prop:rc_space}
$RC(I)$ is a closed and proper linear subspace of $L^\infty(I;\mcl{L}(D;X))$.
\end{prop}

Before stating our Arzelà-Ascoli type result, let us define the notion of relative equicontinuity. 

\begin{df}
We say that a family $\mcl{F} \subset RC(I)$ is \emph{relatively equicontinuous}, if for any $\epsilon > 0$, there exists $\delta > 0$ and $\eta \geq 0$ such that \eqref{eq:def_relative_continuity} holds for any $A \in \mcl{F}$. \\

Equivalently, this means that
\begin{equation*}
\bigcap_{A \in \mcl{F}} r_A^* \neq \emptyset
\end{equation*}

\end{df}

For the statement of the Arzelà-Ascoli-type theorem, we describe several topologies one may put on the space $L^\infty(I;\mcl{L}(D;X))$, here translated in terms of convergence. Given a family $\cbk*{A_\lambda}_{\lambda \in \Lambda} \subset \mcl{L}(D;X)^I$ of nonautonomous operators, where $\Lambda$ is some unspecified metric space, we can describe the following convergences: 

\begin{description}[leftmargin = 50pt]
\item[$(t,x)$-pointwise convergence in $X$:] In this topology $A_\lambda \uset{\lambda \ra \ol{\lambda}}{\lra} A_{\ol{\lambda}}$ if for all $(t,x) \in I\times D$, $A_{\lambda}(t)x \uset{\lambda \ra \ol{\lambda}} \lra A_{\ol{\lambda}}(t)x$ strongly in $X$. It is generated by the product topology of $X^{I\times D}$. 
\item[$t$-pointwise convergence in $\mcl{L}(D;X)$:] In this topology $A_\lambda \uset{\lambda \ra \ol{\lambda}}{\lra} A_{\ol{\lambda}}$ if for all $t \in I$, $A_{\lambda}(t) \uset{\lambda \ra \ol{\lambda}} \lra A_{\ol{\lambda}}(t)$ for the norm topology in $\mcl{L}(D;X)$. It is generated by the product topology of $\mcl{L}(D;X)^I$. 
\item[Uniform convergence:] This is the topology generated by $\norm*{\cdot}_\infty$ as defined in \autoref{df:operator_lebesgue_space}. 
\end{description}

There is an obvious hierarchy of these three topologies, the former being the weakest and the latter being the strongest.

\begin{prop}
\label{prop:arzelà_ascoli}
Let $\mcl{F}\subset RC(I)$ be a family of relatively continuous operators on $I$. The following properties hold:

\begin{enumerate}[label = (\roman*)]
\item \label{i:aa_i} If $\mcl{F}$ is relatively compact in $L^\infty(I;\mcl{L}(D,X))$ then $\mcl{F}$ is relatively equicontinuous. 
\item \label{i:aa_ii} If $\mcl{F}$ is relatively equicontinuous and for each $t \in I$, $\cbk*{A(t) \:\ A \in \mcl{F}}$ is bounded in $\mcl{L}(D,X)$ then $\mcl{F}$ is bounded in $L^\infty(I;\mcl{L}(D,X))$.
\item \label{i:aa_iii} If $\mcl{F}$ is relatively equicontinuous then $\ol{\mcl{F}} \subset RC(I)$ and is also relatively equicontinuous, where $\ol{\mcl{F}}$ is the sequential closure of $\mcl{F}$ for the $(t,x)$-pointwise convergence.
\end{enumerate}
\end{prop}

Note that, contrary to the standard Arzelà-Ascoli theorem for continuous maps, if we replace the two occurrences of "bounded" by the expression "relatively compact" in \ref{i:aa_ii} then the conclusion may fail. A counter-example is provided by the sequence of nonautonomous operators, $n \geq 2$, $t \in [0,1]$, $A_n(t):= A + \1_{[\frac{1}{n}, \frac{2}{n}]}(t) B$ where $A \in \mcl{L}(D,X)$ and $B \in \mcl{L}(D, X)$ satisfies an interpolation inequality i.e. there exists $L > 0$ and $\theta \in (0,1)$ such that for any $x \in D$,

\begin{equation*}
\norm*{Bx}_X \leq L \norm*{x}^{\theta}_X\norm*{x}^{1-\theta}_D
\end{equation*}
Thanks to Young's inequality, it is straightforward to check that the family $\cbk*{A_n}_{n \in \N} \subset RC(0,1)$ is relatively equicontinuous and that $A_n(t)x \lra A$ in $X$ as $n \ra +\infty$ for all $t \in [0,1]$, $x \in D$. However 
\begin{equation*}
\forall n \in \N, \qquad \norm*{A_n - A}_{\infty} = \norm*{B}_{\mcl{L}(D,X)} > 0
\end{equation*}

\subsection{Main results and scheme of the proof}
\label{subsection:main}

In our main result, we derive quantitative estimates of the maximal regularity constant for relatively continuous operators which satisfy pointwise maximal regularity up to a small perturbation. This result is proved through a combination of intermediate results which have their own interest. In \Cref{subsubsection:asymptotic} we describe the asymptotic behavior of the regularity constant under diagonal perturbations (i.e. taking the form $A + \lambda \Id$). In \Cref{subsubsection:maximal} we describe the main quantitative estimate as a result of a gluing-up of operators defined on subintervals. \Cref{subsubsection:topologies} discusses limiting behavior of operators in $\MRprop_p$ in particular its persistence in the limit. Finally in \Cref{subsubsection:global} a general result on the global existence of a solution to the Cauchy problem for abstract quasilinear equations on bounded intervals and a corollary concerning relatively continuous operators in this context are presented. 

\subsubsection{Asymptotic behavior of diagonal perturbations}
\label{subsubsection:asymptotic}
For our quantitative inquiry it becomes important to evaluate quite precisely the behavior of the maximal regularity constant with respect to diagonal perturbations. Indeed, the proof of the main result relies on a perturbation argument. It is stated in the following. 

\begin{thm}
\label{thm:perturbation_dilation}
Let $I\subset \R$ be a bounded open interval and let $A \in L^\infty(I;\mcl{L}(D; X))$ be a nonautonomous operator. The following statements are equivalent: 

\begin{enumerate}[label = $(\roman*)$]
\item  There exists $\lambda_0 \in \R$ such that $A + \lambda_0 \in \MRprop_p(I)$.
\item  $A \in \MRprop_p(I)$.
\item  For any $\lambda \in \R$, $A + \lambda \in \MRprop_p(I)$.
\end{enumerate}
And if one of these statements holds true, we have for any $\lambda \in \R$,

\begin{equation}
\label{eq:bound_nonaut_perturbation_main}
\sbk*{A + \lambda}_{\MRprop_p(I)} \leq c_p(\lambda|I|)\sbk*{A}_{\MRprop_p(I)}
\end{equation}
where the function $c_p: \R \ra \R_+$ satisfies 
\begin{equation}
\label{eq:def_cp_main}
c_p(\nu):= (|\nu| + \kappa_p(\nu))\pr*{1 + e^{p\nu_-} - e^{-p\nu_+}}^{\frac{1}{p}}
\end{equation}
with $\displaystyle \frac{\kappa_p(\nu)}{|\nu|} \uset{|\nu| \ra + \infty}{\lra} 0$ and $\kappa_p(0) = 1$. In particular the following asymptotic behavior is observed for $c_p$: 
\begin{equation*}
c_p(\nu) \uset{\nu \ra -\infty} \sim -\nu e^{-\nu}
\end{equation*}
\begin{equation*}
c_p(\nu) \uset{\nu \ra +\infty} \sim \nu
\end{equation*}

\end{thm}

It is not surprising to see the exponential factor $e^{-\nu}$ appear in the regime $\nu \ra -\infty$ since this exponential behavior is already sharp for autonomous operators (see later \autoref{lem:mrp_resolvent}). It is however still not clear to the authors whether one can discard the extra factor $\nu$ appearing. This asymptotic result is shown in \Cref{section:2}. We can remark \eqref{eq:bound_nonaut_perturbation_main} possesses a particular invariance related to the behavior of $\sbk*{A}$ with respect to the rescaling $A \mapsto \mu A(\mu \cdot)$ (see \autoref{rem:change_var_applications} later). This is a direct consequence of structure of the time-homogeneous norm \eqref{eq:def_mr_local_norm} for $\MR^p(I)$. In fact any estimate in the form
\begin{equation*}
	\sbk*{A + \lambda}_{\MRprop_p(I)} \leq C(\lambda, |I|) \sbk*{A}_{\MRprop_p(I)}
\end{equation*}
can be reduced to an estimate with $\ti{C}(\lambda, |I|):= \inf_{\mu > 0} C(\mu \lambda, \frac{|I|}{\mu})$ thereby forcing the invariance of the inequality with respect to these transformations. 

\subsubsection{Maximal regularity constant of relatively continuous operators}
\label{subsubsection:maximal}
%
%
%

We introduce here several auxiliary quantities which are involved in the estimate of the maximal regularity constant of nonautonomous operators. Let $I \subset \R$ be a bounded open interval and let $A_0: I \mapsto \mcl{L}(D,X)$ satisfy the pointwise autonomous maximal regularity on the unbounded interval $\R_+$: 
\begin{equation}
\label{eq:pointwise_mr_p}
\forall t \in I, \quad A_0(t) \in \mrprop_p(\R_+).
\end{equation}
Then it is well-known that for each $t \in I$, $-A_0(t)$ is the generator of a bounded analytic semi-group $\cbk*{e^{-\tau A_0(t)}}_{\tau \geq 0}$ on $X$ and from \autoref{prop:mr_injections}, we infer $\sbk*{A(t)}_{\mrprop_p(I)} < +\infty$. Hence we can define the following nonnegative maps: 
\begin{equation}
\label{eq:def_KA_MA}
K_0: t \longmapsto \sbk*{A(t)}_{\mrprop_p(I)}, \qquad M_0: t \longmapsto \sup_{\tau \geq 0} \norm{e^{-\tau A_0(t)}}_X
\end{equation}
\begin{equation}
\label{eq:def_epsA}
\epsilon_0: t \longmapsto \frac{1}{2(K_0(t) + 1)(M_0(t) + 1)}
\end{equation}
Now let us describe the main assumption on $A: I \lra \mcl{L}(D,X)$ which guarantees nonautonomous maximal regularity of $A$.

\begin{enumerate}[label = \textbf{(A\arabic*)}, leftmargin = 40pt]
\item \label{item:A1}
$\left\{
\begin{varwidth}{\textwidth}
\begin{enumerate}[label = (\alph*) ]
\item \label{item:a} \textbf{Relative continuity:} $A \in RC(I;(D;X))$.
\item \label{item:b} \textbf{Decomposition:} There exist $A_0: I \lra \mcl{L}(D,X)$ and $B: I \times I \lra \mcl{L}(D,X)$ such that 
\begin{enumerate}[label = (\alph{enumii}.\arabic*), leftmargin = 40pt]
\item \label{item:b1} For all $t, s \in I$, $A(s) = A_0(t) + B(t,s)$.
\item \label{item:b2} For all $t \in I$, $A_0(t) \in \mrprop_p(\R_+)$.
\item \label{item:b3} There exists a range $r_A = (\delta_A, \eta_A) \in r_A^*$ such that for all $x \in D$,
\begin{equation}
\forall t, s\in I, \quad |t-s| \leq \rho_A(t) \quad \Longrightarrow \quad 
\norm*{B(t,s) x}_X \leq \epsilon_0(t) \norm*{x}_D + \mu_A(t) \norm*{x}_X
\end{equation}
where $\rho_A:= \delta_A \circ \epsilon_0$, $\mu_A:= \eta_A \circ \epsilon_0$ and $\epsilon_0$ is defined by \eqref{eq:def_KA_MA}-\eqref{eq:def_epsA}. 
\end{enumerate}
\end{enumerate}
\end{varwidth}
\right.$
\end{enumerate}

Assumption \ref{item:b} says that we can decompose $A$ into the sum of a pointwise maximally regular operator $A_0$ and a \emph{perturbation} $B$, the meaning of this perturbation being expressed by \ref{item:b3}. Note that we may have $B \equiv 0$ in which case \ref{item:A1} amounts to the assumptions of relative continuity and pointwise maximal regularity of $A$ given in \cite[Theorem 2.11]{arendt2007lp-maximal} of Arendt and coauthors. 

\begin{thm}
\label{thm:relative_continuity}
Assume that $A$ satisfies \ref{item:A1} on the bounded open interval $I \subset \R$. Then $A \in \MRprop_p(I)$ and there exists a subdivision $\mcl{T}:= \cbk*{\tau_i}_{0 \leq i \leq N}$ of $\ol{I}$ and center points $\mcl{C}:= \cbk*{t_i}_{1 \leq i \leq N}$ such that 
\begin{gather}
	a = \tau_0 < t_1 < \tau_1 < ... < t_N < \tau_N = b\\
	\rho_A(t_i) \leq \tau_i - \tau_{i-1} \leq 2 \rho_A(t_i), \qquad \forall i \in \cbk*{1,..., N}
\end{gather}
and there exists $K:= K(p, \mcl{T}, \mcl{C}, K_0, M_0, \norm*{A}_\infty) \geq 0$ such that
\begin{equation}
\label{eq:bound_relatively_continuous}
\sbk*{A}_{\MRprop_p(I)} \leq K
\end{equation}
\end{thm}

\paragraph{Quantitative description of $K$}
In order to obtain quantitative growth conditions for the global existence for quasilinear evolution equations, it is of interest to describe the form of $K$ found in \eqref{eq:bound_relatively_continuous}. Define for $(\tau, \mu, M_0, K_0) \in \R_+^4$, the function 
\begin{equation*}
	G(\tau, \mu, M_0, K_0):= 4(M_0 + 1) c_p(-4M\tau\mu) K_0
\end{equation*}
where $c_p$ is given in \autoref{thm:perturbation_dilation}. Given $\mcl{T}$, $\mcl{C}$ as in \autoref{thm:relative_continuity}, for $1 \leq i \leq N$ denote 
\begin{equation*}
	G_i:= G(\tau_i-\tau_{i-1}, \mu_A(t_i), M_0(t_i), K_0(t_i))
\end{equation*}
We show that $K$ in \eqref{eq:bound_relatively_continuous} is written as 
\begin{equation}
\label{eq:quantitative_description_K}
K := \sum_{i = 1}^N \frac{\tau_i - \tau_0}{\tau_i - \tau_{i-1}} \pr*{\prod_{j = i+1}^N H_j(p, \mcl{T},\norm*{A}_\infty, G_j)} G_i
\end{equation}
where for fixed $p, \mcl{T}, C$, the functions $H^j(p,\mcl{T}, C, \cdot): \R_+ \ra \R_+$, $j \in \cbk*{2, ..., N}$ have linear growth at most. Observe in particular that $K$ only depends on the evaluation of the maps $K_0, M_0: I \lra \R_+$ at the center points $\mcl{C} \subset I$, and remark that this dependence is \emph{not} invariant under permutation of the values $\pr*{K_0(t_1), ..., K_0(t_N)}$: this is due to the intrinsic direction of time imposed in the Cauchy problem. The proof of the estimate is found in \Cref{section:4}. \\

We obtain a even more explicit estimate in the particular class of relatively continuous operators which satisfy a Hölder-type regularity and when an additional integrability condition holds for the maps $K_0$ and $M_0$. \\

For $\alpha \in [1, +\infty)$ and $\beta \in [0, +\infty)$, let us introduce the subspace $RC^{\alpha, \beta}(I) \subset RC(I)$.

\begin{df}
\label{df:hölder_relative_continuity}
We say that $A \in RC^{\alpha, \beta}(I)$ if there exists a range of relative continuity $r_A = (\delta_A, \eta_A) \in r_A^*$ for $A$ and $m_\delta \in (0,1]$, $m_\eta \geq 0$ such that
\begin{equation*}
\forall \epsilon > 0, \quad \delta_A(\epsilon) = m_\delta \epsilon^\alpha
\end{equation*}
\begin{equation*}
\forall \epsilon > 0, \quad \eta_A(\epsilon) \leq m_\eta \epsilon^{-\beta}
\end{equation*}
\end{df}

$RC^{\alpha, \beta}(I)$ is a closed linear subspace of $RC(I)$ equipped with the topology induced by $L^\infty(I;\mcl{L}(D,X))$. In many applications the differential operators encountered fall in the class $RC^{\alpha,\beta}(I)$, since it amounts to a Hölder regularity of order $\frac{1}{\alpha}$ in the higher order coefficients and bounded lower-order coefficients. Let $A: I \lra \mcl{L}(D,X)$ satisfy assumption \ref{item:A1} with the decomposition $A_0$ and $B$, for the explicit estimate we ask the additional assumptions:

\begin{enumerate}[label = \textbf{(A\arabic*)}, start = 2, leftmargin = 40pt]
\item \label{item:A2}
$\left\{
\begin{varwidth}{\textwidth}
\begin{enumerate}[label = (\alph*)]
\item \label{item:a_A2} \textbf{$(\alpha, \beta)$-Hölder regularity:} $A \in RC^{\alpha, \beta}(I;(D;X))$ for some $\alpha \geq 1$, $\beta \geq 0$.
\item \label{item:b_A2} $L^{\alpha r}(I)$ \textbf{Integrability condition:} There exists $r > 1$ such that 
\begin{equation*}
\Gamma = \Gamma(\alpha, r, M_0, K_0):= \pr*{\int_I (K_0(t) + 1)^{\alpha r}(M_0(t) + 1)^{\alpha r}\dd t}^{\frac{1}{r}} < +\infty
\end{equation*}
\end{enumerate}
\end{varwidth}
\right.$
\end{enumerate}

\begin{thm}
\label{thm:A_MR_log_estimate}
Assume that $A$ satisfies \ref{item:A1} and \ref{item:A2}. Then there exists a constant $C = C(|I|, \norm*{A}_\infty, p, r, \alpha, \beta)$ such that,
\begin{equation}\label{eq:A_MR_log_estimate}
\log\pr*{\sbk*{A}_{\MRprop_p(I)}} \leq C\pr*{\pr*{\frac{\Gamma}{m_\delta}}^{r^*}
\left\{
\begin{array}{c}
1\\
+\\
m_\eta m_\delta^{s_2(\alpha, \beta, r)} \Gamma^{s_1(\alpha, \beta, r)}\\
+ \\
\displaystyle \log\pr*{\frac{\Gamma}{m_\delta}}
\end{array}
\right\} + m_\eta + 1}
\end{equation}
where $s_1(\alpha, \beta, r):= r^*\pr*{\frac{\beta + 1}{\alpha}-1}$ and $s_2(\alpha, \beta, r):= \frac{\beta + 1}{\alpha + 1} - \pr*{\frac{\alpha + \frac{1}{r}}{\alpha + 1}}s_1(\alpha, \beta, r)$, $r^*:= \frac{r}{r-1}$.
\end{thm}

Estimate \eqref{eq:A_MR_log_estimate} is only suitable to evaluate $\sbk*{A}_{\MRprop_p(I)}$ near $+\infty$ with some precision. Indeed by definition $\Gamma > 0$ therefore, the left-hand side cannot be made smaller than a given positive constant. However the estimate allows the derivation of asymptotic behavior of the constant in terms of the three quantities $m_\delta, m_\eta$ and $\Gamma$. We will see that the asymptotic behaviors have practical uses in existence results for quasilinear equations. 

\subsubsection{Topologies for the persistence of maximal regularity in limits}
\label{subsubsection:topologies}
Let us be given a metric space $(\Lambda, d_\Lambda)$ which describes the range of a given parameter of our operator. Consider a mapping 
\begin{equation*}
A_{(\cdot)}: \left\{
\begin{array}{rcl}
\Lambda & \lra & \mcl{L}(D,X)^{\ol{I}}\\
\lambda & \longmapsto & A_\lambda
\end{array}\right.
\end{equation*}
of nonautonomous operators on $\ol{I}$. We are given an open subset $\mcl{O} \subset \Lambda$ where for each $\lambda \in \mcl{O}$ $A_{\lambda}$ is maximally regular. We are also given a corresponding map of source terms $f_{(\cdot)}: \Lambda \ra L^p(I;X)$. \\

We equip the closure $\ol{\mcl{O}}$ of $\mcl{O}$ with the natural induced topology by $d_\Lambda$, and we denote $\lambda \uset{\mcl{O}}{\ra} \ol{\lambda}$ the corresponding convergence.\\

To study \ref{item:q2}, we rely on the following elementary result. 

\begin{prop}
\label{prop:elem_result}
Assume $\cbk*{A_\lambda}_{\lambda \in \mcl{O}} \subset \MRprop_p(I)$ is such that $\ol{K}:= \sup_{n \in \N} \sbk*{A_n}_{\MRprop_p(I)} < +\infty$. Assume further that there exists $\ol{\lambda} \in \ol{\mcl{O}}$ such that one of the following convergences holds:
\begin{enumerate}[label = (CV$_{\Roman*}$), leftmargin = 70pt]
\item \label{item:unif_cv} $\displaystyle \norm*{A_\lambda- A_{\ol{\lambda}}}_{\infty} \lra 0$ as $\lambda \uset{\mcl{O}}{\ra} \ol{\lambda}$.

\item \label{item:pointwise_cv}
$\displaystyle \forall t \in \ol{I}, \quad \norm*{A_\lambda(t)- A_{\ol{\lambda}}(t)}_{\mcl{L}(D;X)} \lra 0$ as $\lambda \uset{\mcl{O}}{\ra} \ol{\lambda}$. 
\end{enumerate}

Then $A_{\ol{\lambda}} \in \MRprop_p(I)$ and $\sbk*{A_{\ol{\lambda}}}_{\MRprop_p(I)} \leq \ol{K}$.\\

Moreover denote $v_{(\cdot)}: \mcl{O} \lra \MR^p(I)$ the corresponding mapping of solutions to \eqref{eq:cp}$^I_{0, f_{(\cdot)}}$. Then, 
\begin{enumerate}[label = (\roman*)]

\item \label{i:weak_weak} If \ref{item:unif_cv} holds and $f_\lambda \wcv f_{\ol{\lambda}}$ weakly in $L^p(I;X)$ then
\begin{equation*}
v_\lambda \uset{\substack{\lambda \ra \ol{\lambda} \\ \mcl{O}}}{\wcv} v_{\ol{\lambda}} \qquad \text{ weakly in } \MR^p(I)
\end{equation*}

\item \label{i:strong_strong} If \ref{item:pointwise_cv} holds and $f_\lambda \lra f_{\ol{\lambda}}$ strongly in $L^p(I;X)$ then
\begin{equation*}
v_\lambda \uset{\substack{\lambda \ra \ol{\lambda} \\ \mcl{O}}}{\lra} v_{\ol{\lambda}} \qquad \text{ strongly in } \MR^p(I)
\end{equation*}
\end{enumerate}

\end{prop}
Obviously if \ref{item:unif_cv} holds, then \ref{item:pointwise_cv} also holds hence the conclusion also holds. 

We observe the following assumptions on the mapping restricted to $\mcl{O}$:

\begin{enumerate}[label = \textbf{(A$_\Lambda$)}, start = 2, leftmargin = 40pt]
	\item \label{item:Al}
	$\left\{
	\begin{varwidth}{\textwidth}
	\begin{enumerate}[label = (\alph*)]
	\item \label{item:a_Al} For any $t \in \ol{I}$, the family $\cbk*{A_\lambda(t)}_{\lambda \in \mcl{O}} \subset \mrprop_p(\R^*_+)$ and $\sup_{\lambda \in \mcl{O}} \sbk*{A_\lambda(t)}_{\mrprop_p(I)} < + \infty$.
	\item \label{item:b_Al} The family $\cbk*{A_{\lambda}}_{\lambda \in \mcl{O}} \subset RC(I;(D;X))$ and is relatively equicontinuous.
	\item \label{item:c_Al} For any $t \in \ol{I}$, the family $\cbk*{A_\lambda(t)}_{\lambda \in \mcl{O}}$ is bounded in $\mcl{L}(D;X)$. 
	\end{enumerate}
	\end{varwidth}
	\right.$
\end{enumerate}

These assumptions can be used straightforwardly along with \autoref{thm:relative_continuity} to infer the following uniform boundedness of the maximal regularity constant on $\mcl{O}$:

\begin{cor}
\label{cor:uniform_mr_bound}
Assume that $A_{(\cdot)}$ satisfies \ref{item:Al} in $\mcl{O}\subset \Lambda$. Then $\cbk*{A_\lambda}_{\lambda \in \mcl{O}} \subset \MRprop_p(I)$ and we have
\begin{equation*}
\sup_{\lambda \in \mcl{O}}\sbk*{A_\lambda}_{\MRprop_p(I)} < + \infty
\end{equation*}
\end{cor}

The above corollary may be used to infer several results answering \ref{item:q2} affirmatively. The rather elementary proofs of the results below rely on \autoref{prop:arzelà_ascoli}; for the convenience of the reader we chose to collect them in \cite[Appendix C]{belin2024wp}. 

\begin{cor}
\label{cor:weak-weak_stability}
Assume that $A_{(\cdot)}$ satisfies \ref{item:Al}\ref{item:a_Al} in $\mcl{O} \subset \Lambda$ and that there exists $\ol{\lambda} \in \Lambda$ such that \ref{item:unif_cv} holds. 

Then $A_{\ol{\lambda}} \in \MRprop_p(I)$ and \ref{i:weak_weak} of \autoref{prop:elem_result} holds. 
\end{cor}

\begin{cor}
\label{cor:strong-strong_stability}
Assume that $A_{(\cdot)}$ satisfies \ref{item:Al} in $\mcl{O} \subset \Lambda$ and that there exists $\ol{\lambda} \in \Lambda$ such that \ref{item:pointwise_cv} holds. 

Then $A_{\ol{\lambda}} \in \MRprop_p(I)$ and \ref{i:strong_strong} of \autoref{prop:elem_result} holds. 
\end{cor}

\subsubsection{Global existence for a quasilinear  equation}  
\label{subsubsection:global}

We state a global existence result for a class of abstract quasilinear equations. Assume that $\mcl{X}(I) \cinj L^p(I;X)$ is a functional space on $I$, and we make the further assumption that the injection

\begin{equation*}
\MR^p(I) \cinj \mcl{X}(I) \text{ is continuous and compact}
\end{equation*}

Since the injection $D \cinj X$ is compact, we already know that $\MR^p(I)$ injects in $L^p(I;X)$ compactly. Hence there exist such spaces $\mcl{X}(I)$, take for example any real interpolation between $L^p(I;X)$ and $\MR^p(I)$. We consider the following nonlocal-in-time Cauchy problem for a quasilinear equation, 

\begin{equation}
\label{eq:non_local_quasilinear}
\left\{
\begin{array}{rclr}
\displaystyle \ddt u(t) + \mathbb{A}(u)(t)u(t) &=& \mathbb{F}(u)(t) & t\in I\\
\\
u(a) &=& x
\end{array}
\right.
\tag{QL}
\end{equation}
where the initial condition $x \in \Tr^p$. 

To ensure existence of global strong solutions to \eqref{eq:non_local_quasilinear}, i.e. solutions $u \in \MR^p(I)$ which satisfy the equation for a.e. $t \in I$, we introduce the following assumptions on the maps $A_{(\cdot)}: \mcl{X}(I) \lra \mcl{L}(D;X)^{\ol{I}}$ and $F: \mcl{X}(I) \lra L^p(I;X)$: 

\begin{enumerate}[label = (\textbf{E}$_{\Roman*}$), leftmargin = 30pt]
 
\item \label{item:E_I} For all $R \geq 0$, define the quantities: 

\begin{equation}
\gamma(R):= \sup_{\norm*{u}_{\mcl{X}(I)} = R}\sbk*{\mathbb{A}(u)}_{\MRprop_p(I)}
\end{equation}
\begin{equation}
\kappa(R):= \sup_{\norm*{u}_{\mcl{X}(I)} = R}\cbk*{\norm*{\mathbb{A}(u)}_{\infty}  + \norm*{\mathbb{F}(u)}_{L^p(I;X)} + 1}
\end{equation}
and assume that for any $L \geq 0$ the following sublinear growth condition holds,
\begin{equation}
\sup_{R_1, R_2 \in [R-L, R + L]}\frac{\gamma(R_1)\kappa(R_2)}{R} \uset{R \ra +\infty}{\lra} 0 
\end{equation}

\item \label{item:E_II} We ask the following continuity properties:
\begin{enumerate}[label = (\roman*)]
\item $\mathbb{F}: \left\{\begin{array}{rcl} 
\mcl{X}(I) &\lra & L^p(I;X)\\
u &\longmapsto& \mathbb{F}(u)
\end{array}\right.$ is weakly continuous.

\item $\mathbb{A}: \left\{\begin{array}{rcl}
\mcl{X}(I) &\lra &L^\infty(I;\mcl{L}(D,X))\\
u &\longmapsto& \mathbb{A}(u)
\end{array}\right.$ is continuous. 

%
\end{enumerate}

%


\end{enumerate}

The sublinear growth condition stated in \ref{item:E_I} does not allow local singularities or degeneracies to develop in the operator $\mathbb{A}$ and imposes a quite restrictive growth condition on the source term $\mathbb{F}$. Indeed, in the case of a constant operator $\mathbb{A}(u) = A$, it is required that $\mathbb{F}$ have a strictly sublinear growth.\\

The statement of the growth condition involves the supremum over the two variables $R_1$ and $R_2$ in the interval $[R-L, R+L]$. This allows to treat any initial condition $x \in \Tr^p$. One can restrict the range of $L >0$ if one has a smallness assumption on the initial data $x$. In the case of a homogeneous initial data $x = 0$, we can in fact reduce the assumption to $L = 0$ and take $\kappa(R):= \sup_{\norm*{u} = R}\cbk*{\norm*{\mathbb{A}(u)}_\infty + \norm*{\mathbb{F}(u)}_{L^p(I;X)}}$ instead. \\

The existence of strong solutions no longer requires the Lipschitz continuity of the nonlinearities $\mathbb{A}$ and $\mathbb{F}$. 

The method of proof uses Schauder's fixed-point theorem contrary to the proof in \cite{amann2005quasilinear} which uses Banach's fixed point. The uniqueness of the solution is thus not guaranteed with this method.

\begin{thm}
\label{thm:existence_abstract}
Assume that $\mathbb{A}$ and $\mathbb{F}$ satisfy conditions \ref{item:E_I} - \ref{item:E_II}. Then for any $x \in \Tr^p$ there exists a global solution $u \in \MR^p(I)$ of \eqref{eq:non_local_quasilinear}
\end{thm}

This framework is then used to infer quantitative growth conditions on the regularity constant for relatively continuous operators in $RC^{\alpha, \beta}(I)$ which satisfy \ref{item:A1}-\ref{item:A2}. Denote by $\mathbb{A}_0: u \mapsto \mcl{A}_0(u)$ the operator given by the decomposition \ref{item:A1}\ref{item:b} and the subsequent $u \mapsto K_0(u), M_0(u)$ given in \eqref{eq:def_KA_MA}. Naturally this gives rise to the following maps: 

\begin{enumerate}[label = (\roman*), leftmargin = 50pt]
\item $m_\delta: \mcl{X}(I) \lra (0,1]$
\item $m_\eta: \mcl{X}(I) \lra \R_+$
\item $\Gamma: \left\{\begin{array}{ccl}
\mcl{X}(I) & \lra & \R_+\\
u &\longmapsto& \norm*{(K_0(u) + 1)(M_0(u) + 1)}^\alpha_{L^{\alpha r}(I)}
\end{array}\right.$
\end{enumerate}

Our existence criterion reads as follows: 

\begin{enumerate}[label = (\textbf{E}$_I'$)]
\item \label{item:E_I'} For any $R \geq 0$, define the following quantities:
\begin{equation*}
	\gamma_{\log}(R):= \sup_{\norm*{u}_{\mcl{X}(I)} = R} \cbk*{\pr*{\frac{\Gamma(u)}{m_\delta(u)}}^{r^*}\pr*{ 1 + m_\eta(u)m_\delta(u)^{s_2}\Gamma(u)^{s_1} + \log\pr*{\frac{\Gamma(u)}{m_\delta(u)}}} + m_\eta(u)}
\end{equation*}
\begin{equation*}
	\kappa_{\log}(R):= \sup_{\norm*{u}_{\mcl{X}(I)} = R} \cbk*{\log\pr*{\norm*{\mathbb{A}(u)}_\infty + \norm*{\mathbb{F}(u) + 1}_{L^p(I;X)}}}
\end{equation*}
and assume that there exists $h_0 > 0$, such that for any $L > 0$, there exists $R_0 \geq 0$ such that for all $R \geq R_0$ we have
\begin{equation}
\label{eq:growth_eta_delta_Gamma}
\sup_{R_1, R_2 \in [R-L, R + L]}
\frac{\gamma_{\log}(R_1) + \kappa_{\log}(R_2)}{\log\pr*{R}} \leq 1-h_0
\end{equation}
where $s_1(\alpha, \beta, r) =r^*\pr*{\frac{\beta + 1}{\alpha}-1}$ and $s_2(\alpha, \beta, r):= \frac{\beta + 1}{\alpha + 1} - \pr*{\frac{\alpha + \frac{1}{r}}{\alpha +1}}s_1(\alpha, \beta, r)$. 
\end{enumerate}

\begin{thm}
\label{thm:existence_explicit}
Assume that for each $u \in \mcl{X}(I)$, $\mathbb{A}(u)$ satisfies \ref{item:A1} and \ref{item:A2}. Moreover assume that $\mathbb{A}, \mathbb{F}$ satisfy \ref{item:E_Iprime} and \ref{item:E_II}. Then for any $x \in \Tr^p$, there exists a global solution $u$ to \eqref{eq:non_local_quasilinear}. 
\end{thm}

\newpage 

\section{Perturbation and asymptotic behavior}
\label{section:2}

We here study the influence of well-chosen perturbations on the maximal regularity constant. First we recall useful notions of $L^p$ maximal regularity for autonomous operators on unbounded intervals. For such operators, we describe the asymptotic behavior of their maximal regularity constant under diagonal perturbations on bounded intervals. The same prospect is then carried out for nonautonomous, even though much less is known about. Lastly we exhibit the influence of suitable nonautonomous perturbations on the regularity constant. 

\subsection{The operator $L_A:= \ddt + \mcl{A}$}

We assume throughout that $A: \ol{I} \lra \mcl{L}(D;X)$ is a strongly measurable operator. Before our asymptotic study, let us describe the operators of interest in the Cauchy problem \eqref{eq:cp}:
\begin{itemize}
\item $\mcl{A}^I$ is the multiplication operator associated to $A$ on $L^p(I;X)$; it is defined on the domain
\begin{equation*}
D(\mcl{A}^I):= \cbk*{u \in L^p(I;X): t \mapsto A(t)u(t) \in L^p(I;X)}
\end{equation*}

By definition, $\pr*{\mcl{A}^Iu}: t \mapsto A(t)u(t)$ for any $u \in D(\mcl{A}^I)$; 
\item $\displaystyle \ddt$ is the time-derivative operator defined on the domain of $L^p(I;X)$
\begin{equation*}
D\pr*{\ddt} = W^{1,p}_0(I;X):= \cbk*{u \in W^{1,p}(I;X): v(a) = 0}
\end{equation*}
\end{itemize}
%
Their sum is an unbounded operator on $L^p(I;X)$ denoted by

\begin{equation*}
L^I_A:= \ddt + \mcl{A}^I
\end{equation*}
and defined on the domain $D(L^I_A) \supset D(\mcl{A}^I)\cap D(\ddt) = \MR_0^p(I)$ where $\MR_0^p(I)$ is defined in \eqref{eq:def_mr_0}. 

\begin{lem}
\label{lem:mr_la}
$A \in \MRprop_p(I)$ if and only if $L^I_A$ is closed on $D(L^I_A) = \MR^p_0(I)$ and $L^I_A$ is invertible in $\mcl{L}(\MR^p_0(I); L^p(I;X))$. In such a case we have

\begin{equation}
\label{eq:mr_la}
\sbk*{A}_{\MRprop_p(I)} = \norm*{(L^I_A)^{-1}}_{\mcl{L}\pr*{L^p(I;X); \MR^p_0(I)}}
\end{equation}
\end{lem}

\begin{proof}
The proof can be found in \cite{belin2024wp}.

\end{proof}
We can concisely describe 
\begin{equation*}
\MRprop_p(I):= \cbk*{A: I \ra \mcl{L}(D;X),\ \norm*{\cdot}_{A(t)} \sim \norm*{\cdot}_D \text{ for all } t\in I, \ \sbk*{A}_{\MRprop_p(I)} < +\infty}.
\end{equation*}
As before, we shall often write $\mcl{A}$ and $L_A$ instead of $\mcl{A}^I$ and $L_A^I$ when the time interval is clear by context.\\

\subsection{Diagonal perturbations by constants in $\mrprop_p(I)$ and $\MRprop_p(I)$}

Let us show that the maximal regularity property of the operator on a bounded interval $I$ is not affected by a constant diagonal perturbation of the form $\lambda \Id$. We first show it for autonomous operators in $\mrprop_p(R_+)$ (\autoref{lem:mrp_resolvent}) with associated decays and then for nonautonomous operators (\autoref{prop:perturbation_dilation}); in each case we give asymptotic bounds of $\sbk*{A + \lambda I}_{\mrprop_p(I),\MRprop_p(I)}$ when $|\lambda| \ra +\infty$. 

\paragraph{The autonomous case}
Let $A \in \mrprop_p(\R_+)$. Recall (see e.g. \cite[Proposition 2.2]{monniaux2009maximal}) that $-A$ is the generator of an \emph{analytic semi-group} $\pr*{e^{-t A}}_{t \geq 0}$. It means that the semi-group $(e^{-tA})_{t \geq 0}$ can be analytically extended to a sector ${\Sigma_\theta:= \cbk*{z \in \C: |\arg(z)| < \theta} \cup \cbk*{ 0 }}$, $\theta \in (0,\frac{\pi}{2}]$ of the complex plane and $\sup_{z \in \Sigma_\theta}\norm*{e^{-zA}}_{\mcl{L}(X)} < +\infty$. The analyticity of the semi-group is equivalent to the \emph{sectoriality} of the original operator $A$.\\

These facts will prove useful in the study of the resolvent operator $R_\lambda:=
(\lambda I + L_A)^{-1}$ of $-L_A$ for real values of $\lambda$, since this operator can be represented by a Laplace transformation of the semi-group generated by $-L_A$, as stated in the following result.

\begin{lem}
\label{lem:mrp_resolvent_set}
If $A \in \mrprop_p(\R_+)$, then for any bounded open interval $I \subset \R$ and $\lambda\in\R$, $\lambda \Id + L^{I}_A$ is invertible in $\mcl{L}\pr*{L^p(I;X)}$ and we can express 

\begin{equation}
\label{eq:res_laplace}
R_\lambda = \int_0^\infty e^{-\lambda \tau}e^{-\tau L_A} \dd \tau
\end{equation}
\end{lem}

\begin{proof}
  The proof can be found in \cite{belin2024wp}.
\end{proof}
Using representation \eqref{eq:res_laplace}, we further refine our study of $R_\lambda$ by obtaining bounds and asymptotic behaviors when $\lambda \ra \pm \infty$.

\begin{lem}
\label{lem:mrp_resolvent}
Assume that $A \in \mrprop_p(\R_+)$, then for any bounded interval $I \subset \R$ and any $\lambda \in \R$, 
\begin{equation}
\label{eq:res1}
\norm*{R_\lambda}_{\mcl{L}(L^p(I;X))} \leq M\frac{1 - e^{-\lambda |I|}}{\lambda}
\end{equation}
and
\begin{equation}
\label{eq:res2}
\sbk*{A + \lambda \Id}_{\mrprop_p(I)} = \norm*{R_\lambda}_{\mcl{L}(L^p(I;X); \MR^p(I))} \leq \pr*{1 + M|1-e^{-\lambda |I|}|}\sbk*{A}_{\mrprop_p(I)}
\end{equation}
where $M:= \sup_{t \geq 0} \norm*{e^{-tA}}_X$. 
\end{lem}

\begin{proof}

From the uniform boundedness of $\norm*{e^{-z A}}_{\mcl{L}(X)}$ on $\Sigma_\theta$, we obtain the following uniform boundedness (see e.g. \cite[Proposition 2.2]{monniaux2009maximal}),

\begin{equation}
\label{eq:bound_A_cal}
\sup_{\tau \geq 0}\norm*{e^{-\tau \mcl{A}}}_{\mcl{L}(L^p(I;X))} \leq \sup_{\tau \geq 0} \norm*{e^{-\tau A}}_{\mcl{L}(X)} =: M < +\infty.
\end{equation}
  
From \eqref{eq:res_laplace}, since $S_\tau$ is a contraction semi-group and vanishes for $\tau \geq |I|$, we have for any $\lambda \neq 0$ and $v \in L^p(I;X)$, using the triangle inequality and \eqref{eq:bound_A_cal},

\begin{align*}
\norm*{R_\lambda v}_{L^p(I;X)} &\leq \int_0^{|I|} e^{-\lambda \tau} \norm*{e^{-\tau\mcl{A}}v}_{L^p(I;X))}\dd \tau\\
& \leq M\pr*{\int_0^{|I|} e^{-\lambda \tau}\dd \tau} \norm*{v}_{L^p(I;X)}\\
& \leq M\frac{1 - e^{-\lambda |I|}}{\lambda}\norm*{v}_{L^p(I;X)}
\end{align*}
and the above inequality also holds for $\lambda = 0$ if we extend $\lambda \mapsto \frac{1 - e^{-\lambda |I|}}{\lambda}$ by continuity. This yields estimate \eqref{eq:res1}. Observe that $R_\mu = \pr*{L_A + \mu \Id}^{-1} = L_{A + \mu \Id}^{-1}$ for any $\mu \in \R$. Now use the resolvent identity $R_\lambda = R_0 - \lambda R_0R_\lambda = R_0(\Id - \lambda R_\lambda)$ and \eqref{eq:res1} to find \eqref{eq:res2}.
\end{proof}


%

\paragraph{The nonautonomous case}

Similarly we show that a constant diagonal perturbation $\lambda \Id$ does not change the maximal regularity property of a nonautonomous operator on a bounded interval. Moreover we describe the asymptotic behavior of the maximal regularity constant as the diagonal perturbation goes to $\pm\infty$. Parts of the proof techniques are inspired by the proof of \cite[Proposition 2.2]{laasri2013stability}. \\

\begin{thm}
\label{prop:perturbation_dilation}
Let $A: \ol{I} \ra \mcl{L}(D; X)$ be a nonautonomous operator on the bounded open interval $I$. The following statements are equivalent: 

\begin{enumerate}[label = $(\roman*)$]
\item \label{pi} There exists $\lambda_0 \in \R$ such that $A + \lambda_0 \in \MRprop_p(I)$.
\item \label{pii} $A \in \MRprop_p(I)$.
\item \label{piii} For any $\lambda \in \R$, $A + \lambda \in \MRprop_p(I)$.
\end{enumerate}
And if one of these statements holds true, we have for any $\lambda \in \R$,

\begin{equation}
\label{eq:bound_nonaut_perturbation}
\sbk*{A + \lambda }_{\MRprop_p(I)} \leq c_p(\lambda|I|)\sbk*{A}_{\MRprop_p(I)}
\end{equation}
where the function $c_p: \R \ra \R_+$ satisfies 
\begin{equation}
\label{eq:def_cp}
c_p(\nu):= (|\nu| + \kappa_p(\nu))\pr*{1 + e^{p\nu_-} - e^{-p\nu_+}}^{\frac{1}{p}}
\end{equation}
with $\kappa_p(\nu) = o_{\pm \infty}(|\nu|)$ and $\kappa_p(0) = 1$. In particular the following asymptotic behavior is observed for $c_p$: 
\begin{equation*}
c_p(\nu) \uset{\nu \ra -\infty} \sim -\nu e^{-\nu}
\end{equation*}
\begin{equation*}
c_p(\nu) \uset{\nu \ra +\infty} \sim \nu
\end{equation*}

\end{thm}

\begin{rem}
Note that the bounds in the autonomous case \eqref{eq:res2} and nonautonomous case \eqref{eq:bound_nonaut_perturbation} do not have the same asymptotic behaviors in $\nu \ra \pm \infty$ as they differ by a factor $\nu$. However note that $c_p(\nu)$ is sharp around $0$ since $c_p(0) = 1$. It is not yet known to the authors whether or not this asymptotic at $\pm \infty$ for the nonautonomous case is sharp or not. 
\end{rem}

%
%
%
%

\begin{proof}
First see that \ref{piii} $\Rightarrow$ \ref{pii} $\Rightarrow$ \ref{pi} obviously.

Also for $t \in [a,b]$, since $\norm*{\cdot}_{A(t) + \lambda} \sim \norm*{\cdot}_{A(t) + \mu}$ for any $\lambda, \mu \in \R$, we obtain $\norm*{\cdot}_{A(t) + \lambda} \sim \norm*{\cdot}_D$ if and only if $\norm*{\cdot}_{A(t) + \mu} \sim \norm*{\cdot}_D$.\\

There only remains to show \ref{pi} $\Rightarrow$ \ref{piii}. 

\begin{description}
\item[Existence and estimate:] Possibly renaming $A + \lambda_0$ as $A$, let us assume without loss of generality that $\lambda_0 = 0$. Also upon performing a time translation, we can further assume that $I = (0,T)$, with $T = |I|$. \\
Let $\lambda \in \R$, $f \in L^p(0,T;X)$. Define $g: t \mapsto e^{\lambda t} f(t) \in L^p(0,T;X)$, because $A \in \MRprop_p$, there exists $v \in \MR^p(0,T)$ such that 
\begin{equation}
\label{eq:acp_v}
\left\{
\begin{array}{rcl}
\ddt v(t) + A(t)v(t) &=& g(t)\\
v(0)&=& 0
\end{array}
\right.
\end{equation}

And further for any $t \in (0,T]$, using \autoref{prop:mr_injections} \eqref{eq:mr_const_IJ_same_type} we find
\begin{equation}
\label{eq:v_mr}
\norm*{v}_{\MR^p(0,t)} \leq \sbk*{A}_{\MRprop_p(0,T)}\norm*{g}_{L^p((0,t);X)}
\end{equation}
Now let $u: t \mapsto e^{-\lambda t}v(t) \in L^p((0,T);X)$ and for all $t \in [0,T]$ the integrand of the norm of $\MR^p(a,b)$ defined in \eqref{eq:def_mr_local_norm} is ${\sbk*{u}^p(t):= \norm*{u(t)}^p_{X} + \norm*{u(t)}^p_D + T^p\norm*{\ddt u(t)}^p_X \in L^1((0,T);\R_+)}$ . Note that by convexity of $r \mapsto r^p$, for any $\mu \in (0,1)$,

\begin{align*}
\norm*{\ddt u(t)}^p_X &\leq \pr*{\norm*{\lambda e^{-\lambda t}v(t)}_X + \norm*{e^{-\lambda t}\ddt v(t)}_X}^p\\
& = \pr*{\mu \frac{|\lambda|\norm*{v(t)}_X}{\mu} + (1-\mu)\frac{\norm*{\ddt v(t)}_X}{1-\mu}}^pe^{- p\lambda t}\\
& \leq \pr*{\mu^{1-p} |\lambda|^p\norm*{v(t)}^p_X + (1-\mu)^{1-p}\norm*{\ddt v(t)}_X^p}e^{-p\lambda t}
\end{align*}
Therefore we have the bound,
\begin{align}
\sbk*{u}^p(t) &\leq \pr*{\pr*{ 1 + \mu^{1-p} T^p|\lambda|^p}\norm*{v(t)}^p_{X} + \norm*{v(t)}^p_D + (1-\mu)^{1-p}T^p\norm*{\ddt v(t)}^p_{X}}e^{-p\lambda t} \notag \\
& \leq R_p(\mu, \lambda T) e^{-p\lambda t}\sbk*{v}^p(t) \label{eq:R}
\end{align}
where $R_p(\mu, \nu):= \pr*{1 + \mu^{1-p} |\nu|^p} \vee (1-\mu)^{1-p}$. Performing an integration by parts, we obtain the following if $\lambda \neq 0$,

\begin{align*}
\int_0^T e^{-p\lambda t}\sbk*{v}^p(t) \dd t & = e^{-p\lambda T}\int_0^T \sbk*{v}^p(t) \dd t + \underbrace{p\lambda \int_0^T e^{-p\lambda t}\int_0^t \sbk*{v}^p(t)\dd r\dd t}_{\leq 0 \text{ if } \lambda < 0}\\
& \leq \sbk{A}^p_{\MRprop_p}\pr*{e^{-p\lambda T}\norm*{g}^p_{L^p((0,T);X)} + 
p\lambda_+ \int_0^T e^{-p\lambda t}\int_0^t \norm*{g(r)}^p_X\dd r\dd t}
\\
& \leq \sbk{A}^p_{\MRprop_p}\pr*{
e^{-p\lambda T}\norm*{g}^p_{L^p((0,T);X)} + 
p\lambda_+ \int_0^T \norm*{f(r)}_X^p\int_r^T e^{-p\lambda (t-r)}\dd t\dd r
}\\
& \leq \sbk{A}^p_{\MRprop_p}\pr*{
e^{p \lambda_- T}\norm*{f}^p_{L^p((0,T);X)} + 
\frac{p\lambda_+}{p\lambda} \int_0^T \norm*{f(r)}_X^p(1-e^{-p \lambda (T-r)})\dd r
}\\
& \leq \sbk{A}^p_{\MRprop_p} \pr*{
e^{p \lambda_- T}\norm*{f}^p_{L^p((0,T);X)} + 
(1-e^{-\lambda_+ p T})\norm*{f}^p_{L^p((0,T);X)}
}\\
& \leq \sbk*{A}^p_{\MRprop_p}(1 + e^{\lambda_-pT} - e^{-\lambda_+pT})\norm*{f}^p_{L^p((0,T);X)}
\end{align*}
where $\lambda_+:= \lambda \vee 0$ and $\lambda_-:= (-\lambda) \vee 0$. Note that this final inequality also obviously holds if $\lambda = 0$. Now combining with \eqref{eq:R} we find that for any $\mu \in (0,1)$,

\begin{equation}
\label{eq:mrp_u}
\norm*{u}_{\MR^p(0,T)} \leq R_p(\mu, \lambda T)^{\frac{1}{p}}\pr*{1 + e^{\lambda_-pT} - e^{-\lambda_+pT}}^{\frac{1}{p}}\sbk*{A}_{\MRprop_p}\norm*{f}_{L^p(I;X)}
\end{equation}

From \eqref{eq:mrp_u} we infer that $u \in \MR^p(I)$ and solves 
\begin{equation}
\label{eq:nacp_u}
\left\{
\begin{array}{rcll}
\ddt u(t) + A(t)u(t) + \lambda u(t) &=& f(t),& t \in I\\
u(0) &=& 0
\end{array}
\right.
\end{equation}

\item[Uniqueness:] If $\tilde{u} \in \MR^p(I)$ also solves \eqref{eq:nacp_u} then $\ti{v}: t \mapsto e^{\lambda t}\ti{u}(t) \in \MR^p(I)$ must also solve \eqref{eq:acp_v} and by unique solvability, $v = \ti{v}$ hence $u = \ti{u}$. 

\item[Optimization of $R_p(\mu, \nu)$: ] For any $\nu \neq 0$, there exists a unique minimum $\mu_\nu:= \underset{\mu \in (0,1)}{\argmin}\ R_p(\mu, \nu)$. Indeed any minimizer $\mu_\nu$ of $R_p(\cdot,\nu)$ is the unique zero of the continuous and strictly decreasing function $\mu \mapsto 1 + \mu^{1-p}|\nu|^p - (1-\mu)^{1-p}$ which diverges to $+ \infty$ and $-\infty$ at $\mu = 0$ and $\mu = 1$ respectively. Moreover $\nu \mapsto \mu_\nu$ is increasing on $\nu \in \R_+$. \\

Since $(\mu_\nu)_{\nu > 0}$ is bounded from below, by monotonicity, there exists $\mu_0 \in [0,1)$ such that $\mu_\nu \uset{\nu \ra 0}{\lra} \mu_0$. 
Assume by contradiction that $\mu_0 > 0$ then we see that upon taking the limit $\nu \ra 0$, in the identity $0 = 1 + \mu_\nu^{1-p}\nu^p - (1-\mu_\nu)^{1-p}$, we obtain
\begin{align*}
0 &= 1 - (1-\mu_0)^{1-p}\\
1 &= (1-\mu_0)^{1-p}\\
\mu_0 &= 0
\end{align*}
We find a contradiction, meaning $\mu_0 = 0$. Hence $\alpha_p(\nu):= R_p(\mu_\nu, \nu)^{\frac{1}{p}} = (1-\mu_\nu)^{\frac{1-p}{p}} \underset{\nu \ra 0}{\ra} 1$.\\
Since $(\mu_\nu)_{\nu > 0}$ is bounded from above, again by monotonicity, there exists $\mu_\infty \in (0,1]$ such that $\mu_\nu \uset{\nu \ra +\infty}{\lra} \mu_\infty$. And observe that 
\begin{align*}
(1-\mu_\nu)^{1-p} &= 1 + \mu_\nu^{1-p}\nu^p\\
(1-\mu_\nu)^{1-p} &\uset{\nu \ra +\infty}{\lra} +\infty
\end{align*}
Meaning $\mu_\infty = 1$.  We then infer that $\alpha_p(\nu) = \pr*{1 + \mu_\nu^{1-p}\nu^p}^{\frac{1}{p}} \uset{\nu \ra +\infty}{\sim} \mu_\nu^{\frac{1-p}{p}} \nu \sim \nu$.

\end{description}
\end{proof}

We will use this result later to estimate the effect of well-suited nonautonomous perturbations on the maximal regularity constant. 

\subsection{Nonautonomous perturbations of autonomous operators in $\mrprop_p$}

Let us state the following perturbation result, also found in \cite[Proposition 2.7]{arendt2007lp-maximal}. The novelty being that a quantitative estimate of the maximal regularity constant is given. 

\begin{thm}
\label{thm:perturbation}
Let $A_0 \in \mrprop_p(\R_+)$ and $B: \ol{I} \ra \mcl{L}(D; X)$ be strongly measurable on the bounded interval $I$. Assume there exists $\eta \geq 0$, such that for any $t \in I$, $x \in D$,
\begin{equation*}
\norm*{B(t)x}_X \leq \epsilon_0\norm*{x}_D + \eta\norm*{x}_X
\end{equation*}
where $\epsilon_0:= \frac{1}{2(\sbk*{A_0}_{\mrprop_p(I)} + 1)(M_0 + 1)}$ and $M_0 = \sup_{\tau \geq 0}\norm*{e^{-\tau A_0}}_X$. \\

Then $A_0 + B \in \MRprop_p(I)$ and we have the following bound, 

\begin{equation}
\label{eq:regular_perturbation_mr_bound}
\sbk*{A_0 + B}_{\MRprop_p(I)} \leq G(|I|, \eta, M_0, \sbk*{A_0}_{\mrprop_p(I)})
\end{equation}
where $G(\tau, \eta, M, K):= 4(1 + M)c_p(-4M\eta \tau) K$ for $(\tau, \eta, M, K) \in (\R_+)^4$ and $c_p$ is given by \eqref{eq:def_cp}.\\

In such a case we call $B$ a \emph{regular perturbation of} $A_0$.
\end{thm}

\begin{proof}
First see that for any $t \in \ol{I}$, $\norm*{\cdot}_{A_0 + B(t)} \sim \norm*{\cdot}_{A_0} \sim \norm*{\cdot}_D$, since $\epsilon_0 \leq \frac{1}{2}$. 
As before, denote for $\lambda \in \R$, $R_\lambda:= (\lambda + L_A)^{-1}$ and observe that $\lambda + L_{A_0 + B} = (\Id + \mcl{B}R_\lambda)(\lambda + L_A)$  where $\mcl{B}$ is the multiplication operator on $L^p(I;X)$ associated with $B$, this operator is well-defined by strong measurability of $t\mapsto B(t)$. Now for $\lambda > 0$ and using the bounds \eqref{eq:res1} and \eqref{eq:res2} of \autoref{lem:mrp_resolvent} compute the following, 

\begin{align*}
\norm*{\mcl{B}R_\lambda f}_{L^p(I;X)} &\leq \epsilon_0\norm*{R_\lambda f}_{L^p(I;D)} + \eta \norm*{R_\lambda f}_{L^p(I; X)}\\
& \leq \frac{1}{2\pr*{\sbk*{A_0}_{\mrprop_p(I)} + 1}(M_0 + 1)}\norm*{R_\lambda f}_{\MR^p(I)} + \eta \norm*{R_\lambda f}_{L^p(I;X)}\\
& \leq \frac{1}{2}\norm*{f}_{L^p(I;X)} + \eta M\frac{1-e^{-\lambda |I|}}{\lambda} \norm*{f}_{L^p(I;X)}\\
& \leq \pr*{\frac{1}{2} + \frac{\eta M_0}{\lambda}}\norm*{f}_{L^p(I;X)} \end{align*}
Fix $\lambda_0:= 4\eta M$ for which we then have

\begin{equation*}
\norm*{\mcl{B}R_{\lambda_0} f}_{L^p(I;X)} \leq \frac{3}{4}\norm*{f}_{L^p(I;X)}
\end{equation*}
Hence $\Id + \mcl{B}R_{\lambda_0}$ is an invertible operator of $L^p(I;X)$, and so $L_{A_0 + B + \lambda_0} = (\Id + \mcl{B}R_{\lambda_0})(\lambda_0 + L_{A_0})$ is invertible as an operator of $\mcl{L}(\MR^p(I), L^p(I;X))$. By virtue of \eqref{eq:res2} in \autoref{lem:mr_la} we infer that $A_0 + B + \lambda_0 \in \MRprop_p(I)$ with

\begin{align*}
\sbk*{A_0 + B + \lambda_0}_{\MRprop_p(I)} &\leq \pr*{1-\frac{3}{4}}^{-1}(1+M)\sbk*{A_0}_{\mrprop_p(I)}\\
 &= 4(1+M)\sbk*{A_0}_{\mrprop_p(I)}
\end{align*}



Apply \autoref{prop:perturbation_dilation} with $-\lambda_0$ to obtain that $A_0 + B = (A_0 + B + \lambda_0) - \lambda_0\in \MRprop_p(I)$ with the desired bound \eqref{eq:regular_perturbation_mr_bound}.
\end{proof} 

This perturbation result is key to the analysis of the maximal regularity property of nonautonomous operators which are relatively continuous.

\section{Gluing maximally regular operators}
\label{section:3}

As a tool for the quantitative study of regularity constant of relatively continuous operators, let us derive an estimate concerning the gluing of finite family $(A_i)_{1 \leq i \leq N}$ of maximally regular operators defined on a subdivision of the bounded time interval $I = (a,b)$.

\subsection{Behavior of $\sbk*{\cdot}_{\MRprop_p(I)}$ under a change of variable in time}

To refine the bounds on the maximal regularity constants obtained in the next section, it is useful to study the behavior of $\sbk*{\cdot}_{\MRprop_p(I)}$ with respect to a change of variables in time. An invariance stemming from the choice of the norm of $\MR^p(I)$ \eqref{eq:def_mr_local_norm} is observed when the change of variables is linear (see \autoref{rem:change_var_applications} later).

\begin{lem}
\label{lem:time_change_of_variable}
Let $I, J \subseteq \R$ be two bounded open intervals and $\phi: I \rightarrow J$ be a $C^1$- diffeomorphism and denote $\psi:= \phi^{-1}$. Let $A: J \rightarrow \mcl{L}(D, X) \in \MRprop_p(J)$, then the nonautonomous operator $A_\phi$ defined by 
\begin{equation*}
A_\phi = \phi'(\cdot) \pr*{A\circ \phi}(\cdot): I \rightarrow \mcl{L}(D,X)
\end{equation*}
belongs to $\MRprop_p(I)$ and we have, 

\begin{equation}
\label{eq:mr_const_Aphi}
\sbk*{A_\phi}_{\MRprop_p(I)} \leq \norm*{\psi'}_\infty \max\pr*{1, \frac{|I|}{|J|}\norm*{\phi'}_{\infty}} \sbk*{A}_{\MRprop_p(J)}
\end{equation}

\end{lem}

\begin{proof}
Before proving inequality \eqref{eq:mr_const_Aphi} we first estimate the continuity constant of the isomorphism $\iota_\phi: u \mapsto u\circ \phi$ from $\MR^p(J)$ to $\MR^p(I)$ induced by $\phi$. \\

Let $u \in \MR^p(J)$ and let us estimate $\norm*{\iota_\phi(u)}_{\MR^p(I)}$. 
\begin{description}
\item[Estimation of the $L^p(I;X)$ and $L^p(I;D)$ norms:] We compute
\begin{align}
	\norm*{u\circ \phi}^p_{L^p(I;X)} & = \int_I \norm*{u (\phi(t))}^p_X \dd t\notag \\
	& = \int_J \norm*{u(s)}^p_X |\psi'(s)|\dd s \notag \\
 	& \leq \norm*{\psi'}_{\infty} \norm*{u}^p_{L^p(J;X)}.\label{eq:iota_Lp_X} 
\end{align}
Applying the previous computations with the Banach space $D$, we obtain the similar result
\begin{equation}
\label{eq:iota_Lp_D}
\norm*{u\circ \phi}^p_{L^p(I;D)} \leq \norm*{\psi'}_{\infty} \norm*{u}^p_{L^p(J;D)}
\end{equation}

\item[Estimation of the $W^{1,p}(I;X)$ semi-norm:]
\begin{align}
	|I|^p\norm*{\ddt \pr*{u\circ \phi}}^p_{L^p(I;X)} & = |I|^p\int_I \norm*{\phi'(t) \pr*{\frac{\dd}{\dd s} u}(\phi(t))}^p_X \dd t \notag \\
	& = |I|^p\int_J \snorm*{\phi'(\psi(s))}^p\norm*{\frac{\dd}{\dd s}u(s)}^p_X \snorm*{\psi'(s)}\dd s \notag \\
	& \leq \norm*{\psi'}_{\infty}\pr*{\frac{|I|}{|J|}\norm*{\phi'}_{\infty}}^p |J|^p \norm*{\frac{\dd}{\dd s} u}^p_{L^p(J;X)}. \label{eq:iota_W1p_X}
\end{align}

\item[Operator norm of $\iota_\phi$:] By combining inequalities \eqref{eq:iota_Lp_X}, \eqref{eq:iota_Lp_D} and \eqref{eq:iota_W1p_X} we obtain a bound on the operator norm $\norm*{\iota_\phi}_{\mcl{L}(\MR^p(J),\MR^p(I))}$ of $\iota_\phi$:

\begin{equation}
\label{eq:op_norm_iota_phi}
\norm*{\iota_\phi} \leq \norm*{\psi'}_{\infty}^{\frac{1}{p}} \max\pr*{1, \frac{|I|}{|J|} \norm*{\phi'}_{\infty}}
\end{equation}

Inequality \eqref{eq:op_norm_iota_phi} becomes an identity in the case of affine homomorphisms (i.e. $\phi$ of the form $\phi(t) = \lambda t + t_0$), since inequalities \eqref{eq:iota_Lp_X} - \eqref{eq:iota_W1p_X} saturate and $\frac{|I|}{|J|} \norm*{\phi'}_{\infty} = 1$. Remark that if $\phi' \geq 0$, then $\iota_\phi$ also induces an isomorphism of $\MR^p_0(J)$ onto $\MR^p_0(I)$ with the same operator norm. \\

\item[Maximal regularity of $A_\phi$:] Denote $a_I:= \inf I$, $a_J:= \inf J$ and let $f \in L^p(I;X)$. By the isomorphism $\iota_\phi$, the unique solvability in $\MR^p_0(I)$ of the nonautonomous Cauchy problem 

\begin{equation}
\label{eq:nacp_Aphi}
\left\{
\begin{array}{rcll}
\displaystyle \frac{\dd}{\dd t} v + A_\phi v &=& f,& \text{ in } I\\
v(a_I) &=& 0
\end{array}
\right.
\end{equation}
is equivalent to the unique solvability in $\MR^p_0(J)$ of 

\begin{equation}
\label{eq:nacp_A}
\left\{
\begin{array}{rcll}
\displaystyle \frac{\dd}{\dd s} u + A u &=& \displaystyle \frac{1}{\phi'\circ \psi} f\circ\psi,& \text{ in } J\\
u(a_J) &=& 0.
\end{array}
\right.
\end{equation}
Since ${\norm*{\frac{1}{\phi'\circ \psi} f\circ\psi}_{L^p(J;X)} \leq \norm*{\psi'}_{\infty}^{1 - \frac{1}{p}} \norm*{f}_{L^p(I;X)}} < +\infty$, the assumption $A \in \MRprop_p(J)$ shows existence and uniqueness of $v \in \MR^p_0(I)$ solution of \eqref{eq:nacp_Aphi}. Let us then estimate $\sbk*{A_\phi}_{\MRprop_p(J)}$:

\begin{align*}
\norm*{v}_{\MR^p(I)} &\leq \norm*{\iota_\phi}\norm*{u}_{\MR^p(J)}\\
& \leq \norm*{\iota_\phi}\sbk*{A}_{\MRprop_p(J)}\norm*{\frac{1}{\phi'\circ\psi} f\circ \psi}_{L^p(J,X)}\\
& \leq \norm*{\iota_\phi}\sbk*{A}_{\MRprop_p(J)} \norm*{\psi'}_{\infty}^{1-\frac{1}{p}}\norm*{f}_{L^p(I;X)}\\
& \leq \norm*{\psi'}_{\infty} \max\pr*{1, \frac{|I|}{|J|} \norm*{\phi'}_{\infty}}\sbk*{A}_{\MR^p(J)} \norm*{f}_{L^p(I;X)}
\end{align*}

From which we infer \eqref{eq:mr_const_Aphi}.

\end{description}
\end{proof}

As a direct application of \autoref{lem:time_change_of_variable} we derive the two following important remarks for any $A \in \MRprop_p(J)$:

\begin{rem}
\label{rem:change_var_applications}
\begin{enumerate}
\item Note that $\pr*{A_\phi}_\psi = A$, hence using \eqref{eq:mr_const_Aphi} with $\psi$ we obtain a lower bound for $\sbk*{A_\phi}_{\MRprop_p(I)}$ as 
\begin{equation}
\label{eq:lower_bound_mr_const_Aphi}
\frac{1}{\norm*{\phi'}_\infty}\min\pr*{1, \frac{|I|}{|J|}\frac{1}{\norm*{\psi'}_{\infty}}} \sbk*{A}_{\MRprop_p(J)} \leq \sbk*{A_\phi}_{\MRprop_p(I)}
\end{equation}

\item Let $\phi: I \mapsto J$ be a linear homomorphism, denote $\phi'= \lambda := \frac{|J|}{|I|}$ and $A_\lambda := A_\phi$. Using both \eqref{eq:mr_const_Aphi} and \eqref{eq:lower_bound_mr_const_Aphi} we find
\begin{equation*}
\frac{1}{\lambda}\sbk*{A}_{\MRprop_p(J)} \leq \sbk*{A_\lambda}_{\MRprop_p(I)} \leq \frac{1}{\lambda}\sbk*{A}_{\MRprop_p(J)},
\end{equation*}
yielding 
\begin{equation}
\label{eq:mr_const_Aphi_affine}
\sbk*{A_\lambda}_{\MRprop^p(I)} = \frac{1}{\lambda}\sbk*{A}_{\MRprop_p(J)}
\end{equation}
\end{enumerate}

\end{rem}

\subsection{Gluing finitely many $\MRprop_p$ operators}

We here show that finitely many adjacent maximally regular operators can be glued-up together to form a maximally regular operator on the joined interval. The results are all stated as quantitatively as possible with the tools developed before. The estimations here described are new in the theory of $L^p$ maximal regularity. \\

A first lemma is derived on the gluing of functions in $\MR^p$. Working with the weighted $W^{1,p}$ semi-norm in \eqref{eq:def_mr_local_norm} yields nonstandard norm behavior with respect to restriction and extension of functions in $\MR^p$. 

\begin{lem}
\label{lem:glue_functions_mrp}
Let $v_1 \in \MR^p(a,b)$, $v_2 \in \MR^p(b,c)$, with $ a < b < c \in \R$ and assume that $v_1(b) = v_2(b)$. Then $v:= v_1\1_{[a,b]} + v_2\1_{[b,c]} \in \MR^p(a,c)$ and we have

\begin{equation}
\label{eq:glue_functions_mrp}
\norm*{v}_{\MR^p(a,c)} \leq \kappa_1\norm*{v_1}_{\MR^p(a,b)} + \kappa_2\norm*{v_2}_{\MR^p(b,c)}
\end{equation}
where $\kappa_1 = \frac{c-a}{b-a}$, $\kappa_2:= \frac{c-a}{c-b}$. 
\end{lem}

\begin{proof}
Let $\xi \in C^1_c((a,c);\R)$ and see
\begin{align*}
\int_a^c \pr*{\ddt \xi(t)}v(t)\dd t &= \int_a^b \pr*{\ddt \xi(t)}v_1(t)\dd t + \int_b^c \pr*{\ddt \xi(t)}v_2(t)\dd t\\
&= \sbk*{\xi(t)v_1(t)}_a^b - \int_a^b \xi(t)\ddt v_1(t)\dd t + \sbk*{\xi(t)v_2(t)}_b^c - \int_b^c \xi(t)\ddt v_2(t)\dd t\\
&= \underbrace{\xi(b)v_1(b) - \xi(b)v_2(b)}_{ = 0} - \int_a^c \xi(t)\pr*{\ddt v_1(t)} \1_{[a,b]}(t) + \pr*{\ddt v_2(t)}\1_{[b,c]}(t) \dd t\\
&= - \int_a^c \xi(t)\pr*{\ddt v_1(t)} \1_{[a,b]}(t) + \pr*{\ddt v_2(t)}\1_{[b,c]}(t) \dd t
\end{align*}
Therefore $\ddt v = \pr*{\ddt v_1} \1_{[a,b]} + \pr*{\ddt v_2}\1_{[b,c]} \in L^p((a,c);X)$. Since also $v \in L^p((a,c);D)$ we infer that $v \in \MR^p(a,c)$. Furthermore observe that
\begin{align*}
\norm*{v}^p_{\MR^p(a,c)} &= \norm*{v}^p_{L^p((a,c);X)} + \norm*{v}^p_{L^p((a,c);D)} + (c-a)^p\norm*{\ddt v}^p_{L^p((a,c);X)}\\
&= \norm*{v_1}^p_{L^p((a,b);X)} + \norm*{v_2}^p_{L^p((b,c);X)} + \norm*{v_1}^p_{L^p((a,b);D)} + \norm*{v_2}^p_{L^p((b,c);D)} \\
&\quad + (c-a)^p \pr*{\norm*{\ddt v_1}^p_{L^p((a,b);X)} + \norm*{\ddt v_2}^p_{L^p((b,c);X)}} \\ 
&\leq \pr*{\frac{c-a}{b-a}}^p\norm*{v_1}^p_{\MR^p(a,b)} + \pr*{\frac{c-a}{c-b}}^p\norm*{v_2}^p_{\MR^p(b,c)}
\end{align*}
Now see that $\pr*{\alpha^p + \beta^p}^{\frac{1}{p}} \leq \alpha + \beta$ for any $\alpha, \beta \geq 0$ to conclude \eqref{eq:glue_functions_mrp}.
\end{proof}

As a first step, we glue two operators together and then proceed by induction to show maximal regularity and bound on the regularity constant for finitely many operators. Some asymptotic behaviors are discussed as well as the sharpness of obtained estimate. \\

\begin{lem}
\label{lem:glueing}
Assume that $A_1 \in \MRprop_p(a,b)$ and $A_2 \in \MRprop_p(b,c)$ for some $a<b<c \in \R$. Then the \emph{glued-up} operator $A$ defined by
\begin{equation*}
A(t):= \left\{
\begin{array}{lr}
A_1(t), & a \leq t < b\\
A_2(t), &  b \leq t \leq c
\end{array}
\right.
\end{equation*}
belongs to $\MRprop_p(a,c)$ and we have

\begin{equation}
\label{eq:glueing_A}
\sbk*{A}_{\MRprop_p(a,c)} \leq \kappa_1\sbk*{A_1} + \kappa_2\sbk*{A_2} +\kappa_1^\frac{1}{p}\kappa_2^\frac{1}{q}Q_p\pr*{c-b, \norm*{A_2}_{\infty}, \sbk*{A_2}}\sbk*{A_1}.
\end{equation}
where $\kappa_1:= \frac{c-a}{b-a}$ and $\kappa_2:= \frac{c-a}{c-b}$. Moreover denoting $w_p:= 2^{\frac{1}{q}}(p-1)$ we have
\begin{equation}
\label{eq:def_Q_p}
Q_p(T, C, G):=\left\{
\begin{array}{ll}
1 + 2^\frac{1}{q}\pr*{CG \vee \frac{G}{T}} & \text{ if } w_p \frac{G}{T}\leq 1\\
\\
\displaystyle \frac{2^{\frac{1}{pq}}p}{\pr*{p-1}^{\frac{1}{q}}}\pr*{\frac{G}{T}}^\frac{1}{p} & \text{ if } 1 \leq w_p \frac{G}{T} \leq \frac{1}{CT}\\ 
\\
\pr*{\frac{1}{CT}\vee 1}^\frac{1}{p}\pr*{1 + 2^\frac{1}{q}CG} & \text{ if }\frac{1}{CT} \leq w_p \frac{G}{T} 
\end{array}
\right.
\end{equation}

\end{lem}

\begin{proof}
\begin{itemize}
\item[$\star$] Remark that $A$ obviously satisfies \ref{df:NAi} by maximal regularity of $A_1$ and $A_2$. 

\item[$\star$]Let $f \in L^p((a,c);X)$, let us show that there exists a unique $v \in \MR^p(a,c)$ satisfying \eqref{eq:cp}$^{(a,c)}_{0,f}$ and that \eqref{eq:MR_ineq} holds. 

\begin{description}

\item[Uniqueness:] Assume that $v$ is a solution to \eqref{eq:cp}$^{(a,c)}_{0,f}$. It implies that $v|_{[a,b]}$ is a solution of the Cauchy problem with operator $A_1$ and source term $0$, hence $v = 0$ on $[a,b]$. Now this implies that $v|_{[b,c]}$ starts at $0$ and is a solution of the Cauchy problem with operator $A_2$ and source term $0$, again $v = 0$ on $[b,c]$.

\item[Existence and initial estimate:] Denote $f_1:= f|_{[a,b]} \in L^p((a,b);X)$, $f_2 = f|_{[b,c]} \in L^p((b,c);X)$. For convenience of notations let us set $\sbk*{A_1}:= \sbk*{A_1}_{\MRprop_p(a,b)}$ and $\sbk*{A_2}:= \sbk*{A_2}_{\MRprop_p(b,c)}$.\\

Since $A_1 \in \MRprop_p(a,b)$, there exists a (unique) solution $v_1 \in \MR^p(a,b)$, solution of the Cauchy problem starting at $0$ with source term $f_1$ and such that
\begin{equation}
\label{eq:bound_v1}
\norm*{v_1}_{\MR_p(a,b)} \leq \sbk*{A_1}\norm*{f_1}_{L^p((a,b);X)}
\end{equation}
Denote $x = v_1(b) \in \Tr^p$. Fix $\mu \geq 1$ and let $z \in \cbk*{w \in \MR^p(b,b + \mu(c-b)): w(b) = x}$, which is nonempty since $x \in \Tr^p$. By definition, $L_{A_2} z|_{(b,c)} \in L^p((b,c);X)$, let us then denote $g_z:= f_2 - L_{A_2}z|_{[b,c]} \in L^p((b,c);X)$. Since $A_2 \in \MRprop_p(b,c)$ there exists $v_z \in \MR^p(b,c)$ such that,
\begin{equation*}
\left\{
\begin{array}{rcll}
\ddt v_z(t) + A_2(t) v_z(t) &=& g_z(t), & \text{ a.e. } t \in (b,c)\\
v_z(b) &=& 0&
\end{array}
\right.
\end{equation*}
And by linearity $v_2:= v_z + z|_{[b,c]}$ must be the (unique) solution to 
\begin{equation*}
\left\{
\begin{array}{rcll}
\ddt v_2(t) + A_2(t) v_2(t) &=& f_2(t), & \text{ a.e. } t \in (b,c)\\
v_2(b)&=& x
\end{array}
\right.
\end{equation*}
Now by maximal regularity of $A_2$ we infer:
\begin{align*}
\norm*{v_z}_{\MR^p(b,c)} &\leq \sbk*{A_2}\norm*{g}_{L^p((b,c);X)}\\
& \leq \sbk*{A_2}\norm*{f_2}_{L^p((b,c); X)} + \sbk*{A_2}\norm*{L_{A_2} z}_{L^p((b,c); X)}
\end{align*}
Observe that,
\begin{align}
\norm*{L_{A_2} z}^p_{L^p((b,c); X)} &= \int_b^c \norm*{\ddt z(t) + A_2(t)z(t)}^p_X \dd t \notag \\
& \leq 2^{p-1} \pr*{\int_b^c \norm*{\ddt z(t)}^p\dd t + C^p \int_b^c \norm*{z(t)}^p_D \dd t} \notag \\
& \leq 2^{p-1}\pr*{\int_b^{b + \mu(c-b)}\norm*{\ddt z(t)}^p\dd t + C^p \int_b^{b + \mu(c-b)} \norm*{z(t)}^p_D \dd t} \notag \\
& \leq 2^{p-1}\pr*{C^p \vee \frac{1}{\mu^p(c-b)^p}}\norm*{z}^p_{\MR^p(b,b + \mu(c-b))} \label{eq:bound_LA2}
\end{align}
according to \eqref{eq:def_mr_local_norm} of \autoref{df:mrp_space} and where $C:= \norm*{A_2}_{L^\infty((b,c);\mcl{L}(D,X))} < +\infty$. Therefore we can infer the following bound using \eqref{eq:bound_LA2}:
\begin{align*}
\norm*{v_2}_{\MR^p(b,c)} &\leq \norm*{v_z}_{\MR^p(b,c)} + \norm*{z}_{\MR^p(b,c)}\\
&\leq  \sbk*{A_2}\norm*{f_2}_{L^p((b,c);X)} + 2^{\frac{1}{q}}\pr*{C \vee \frac{1}{\mu(c-b)}}\sbk*{A_2}\norm*{z}_{\MR^p(b,b + \mu(c-b))} + \norm*{z}_{\MR^p(b,b + \mu(c-b))}\\
& \leq \sbk*{A_2}\norm*{f_2}_{L^p((b,c);X)} + \pr*{1 + 2^{\frac{1}{q}}\pr*{C \vee \frac{1}{\mu(c-b)}}\sbk*{A_2}}\norm*{z}_{\MR^p(b,b + \mu(c-b))}
\end{align*}
Passing to the infimum over $z \in \cbk*{w \in \MR^p(b,b + \mu(c-b)) \:\ w(b) = x}$, we find
\begin{equation*}
\norm*{v_2}_{\MR^p(b,c)} \leq \pr*{1 + 2^{\frac{1}{q}}\pr*{C \vee \frac{1}{\mu(c-b)}}\sbk*{A_2}}\norm*{x}_{\Tr^p, (b,b + \mu(c-b))} + \sbk*{A_2}\norm*{f_2}_{L^p(b,c;X)}
\end{equation*}
Using the rescaling of the norms on the trace space described in \autoref{rem:trace_space} together with \eqref{eq:bound_v1}, we find
\begin{align*}
\norm*{x}_{\Tr^p, (b,b + \mu(c-b))} &= \pr*{\mu\kappa_*}^\frac{1}{p} \norm*{x}_{\Tr^p,(a,b)} \\
&\leq \pr*{\mu\kappa_*}^\frac{1}{p}\norm*{v_1}_{\MR^p(a,b)} \\
&\leq \pr*{\mu\kappa_*}^\frac{1}{p}\sbk*{A_1}\norm*{f_1}_{L^p((a,b);X)}
\end{align*}
where ${\displaystyle \kappa_*:= \frac{c-b}{b-a}}$, therefore we finally obtain, 

\begin{equation}
\label{eq:v2_bound}
\norm*{v_2}_{\MR^p(b,c)} \leq (\mu\kappa_*)^{\frac{1}{p}}\pr*{1 + 2^{\frac{1}{q}}\pr*{C \vee \frac{1}{\mu(c-b)}}\sbk*{A_2}}\sbk*{A_1}\norm*{f_1}_{L^p((a,b);X)} + \sbk*{A_2}\norm*{f_2}_{L^p((b,c);X)}
\end{equation}
Define the function $v:= v_1\1_{[a,b]} + v_2\1_{[b,c]}$. Since $v_1(b) = v_2(b)$, we infer from \autoref{lem:glue_functions_mrp} that $v \in \MR^p(a,c)$. Further we observe that $v$ is in fact a solution to \eqref{eq:cp} for a.e. $t \in [a,c]$. Moreover recall that 

\begin{equation}
\label{eq:glueing_v1_v2}
\norm*{v}_{\MR^p(a,c)} \leq \kappa_1\norm*{v_1}_{\MR^p(a,b)} + \kappa_2\norm*{v_2}_{\MR^p(b,c)} 
\end{equation}
where $\kappa_1 = \frac{c-a}{b-a}$ and $\kappa_2 = \frac{c-a}{c-b}$. Remark that $\frac{1}{\kappa_1} + \frac{1}{\kappa_2} = 1$ and $\kappa_* = \frac{\kappa_1}{\kappa_2}$. Hence assembling \eqref{eq:v2_bound} and \eqref{eq:glueing_v1_v2}, together with the fact that $\norm*{f_1}_{L^p((a,b);X)}, \norm*{f_2}_{L^p((b,c);X)} \leq \norm*{f}_{L^p((a,c);X)}$:

\begin{equation*}
\norm*{v}_{\MR^p(a,c)} \leq \pr*{\kappa_1\sbk*{A_1} + \kappa_2\pr*{(\mu\kappa_*)^{\frac{1}{p}}\pr*{1 + 2^{\frac{1}{q}}\pr*{C \vee \frac{1}{\mu(c-b)}}\sbk*{A_2} }\sbk*{A_1} +  \sbk*{A_2}}}\norm*{f}_{L^p((a,c);X)}
\end{equation*}
By remarking that $\kappa_2(\kappa_*)^\frac{1}{p} = \kappa_1^\frac{1}{p}\kappa_2^\frac{1}{q}$ we deduce the following estimate of $\sbk*{A}_{\MRprop_p(a,c)}$: 

\begin{equation}
\label{eq:first_bound_mrA}
\sbk*{A}_{\MRprop_p(a,c)} \leq \kappa_1\sbk*{A_1} + \kappa_2\sbk*{A_2} + \kappa_1^\frac{1}{p}\kappa_2^\frac{1}{q}\mu^{\frac{1}{p}}\pr*{1 + 2^{\frac{1}{q}}\pr*{C \vee \frac{1}{\mu(c-b)}}\sbk*{A_2}}\sbk*{A_1}
\end{equation}
\item[Optimization with $\mu$:]
Let us denote $\ti{Q}_p(T, C, G, \mu):= \mu^\frac{1}{p}\pr*{1 + 2^{\frac{1}{q}}\pr*{C \vee \frac{1}{\mu T}}G}$. Let $G, C, T$ be given and let us minimize $\mu \mapsto P(\mu):= \ti{Q}_p(T, C, G, \mu)$ on $[1, +\infty)$. Remark that $P(\mu) = \alpha(\mu) \vee \beta(\mu)$, where 
\begin{equation*}
\alpha(\mu):= \mu^\frac{1}{p}\pr*{1 + 2^{\frac{1}{q}}C G}
\end{equation*}
is increasing and
\begin{equation*}
\beta(\mu):= \mu^\frac{1}{p}\pr*{1 + \frac{2^{\frac{1}{q}}G}{\mu T}}
\end{equation*}
We note that $\mu \mapsto \beta(\mu)$ admits a unique global minimum on $\R^*_+$ attained at $\mu_1:= (p-1)2^{\frac{1}{q}}\frac{G}{T} =: w_p\frac{G}{T}$, because $\beta$ is decreasing on $(0,\mu_1]$ and increasing on $[\mu_1, +\infty)$. Moreover observe that $\alpha$ and $\beta$ both meet at a unique point $\mu_2 = \frac{1}{CT}$. Therefore the global minimum of $P(\mu)$ on $\R_+^*$ is attained at $\mu_*:= \mu_1 \wedge \mu_2$. We can therefore infer that the minimum of $P$ on $[1, +\infty)$ is attained at $1 \vee \pr*{\frac{1}{CT} \wedge w_p\frac{G}{T}}$ which yields our result. 

\end{description}

\end{itemize}

\end{proof}

\begin{rem}
Note that in \eqref{eq:glueing_A}, the roles of $A_1$ and $A_2$ are not symmetric, this accounts for the fact that the maximal regularity constant is defined with the initial condition $x = 0$. It is worth noting that it is the multiplicative interaction between $\sbk*{A_1}$ and $\sbk*{A_2}$ of the term $q\kappa_1\kappa_2Q_p(c-b, C, \sbk*{A_2})\sbk*{A_1}$ which gives rise to the asymmetry.\\

It might seem possible to further refine \eqref{eq:glueing_A} by exploiting a piecewise linear change of variable $\phi$ such that $|\phi((a,b))| = |\phi((b,c))|$ so as to obtain optimality on the factors $\kappa_1$ and $\kappa_2$. Unfortunately the factor in \eqref{eq:mr_const_Aphi} cancels out the optimality. 
\\
Remark that for any $\lambda > 0$, $C,T,G \in \R_+$, we have the invariance $Q_p(C, T, G) = Q_p(\lambda C, \frac{T}{\lambda}, \frac{G}{\lambda})$. This invariance is compatible with a linear change of variables in time (see \autoref{rem:change_var_applications} \eqref{eq:mr_const_Aphi_affine}). This then implies that \eqref{eq:glueing_A} is invariant with respect to a linear change of variables in time. 
\end{rem}

Now we derive the result for the glued-up operator of a finite family of nonautonomous operators. We first define the following geometrical parameters related to the subdivision. 

\begin{df}
\label{df:parameter_family}
Let $N \geq 2$ and $\mcl{T}_N = \cbk*{\tau_i}_{0 \leq i \leq N}$ be such that $a = \tau_0 < \tau_1 < ... < \tau_N = b$, define for each $j \in \cbk*{2, ..., N}$ the function $H_j(p,\mcl{T}, \cdot, \cdot): \R_+\times \R_+ \to \R_+$ as

\begin{equation}
\label{df:H_j}
H_j(p,\mcl{T}, C, G):= \kappa_{1,j} + \kappa_{1,j}^{\frac{1}{p}}\kappa_{2,j}^\frac{1}{q}Q_p\pr*{\tau_j - \tau_{j-1}, C, G}
\end{equation}

where $\kappa_{1, j}:= \frac{\tau_{j} - \tau_0}{\tau_{j-1} - \tau_0}$, $\kappa_{2, j}:= \frac{\tau_j - \tau_0}{\tau_j - \tau_{j-1}}$.
\end{df}

We now give an estimate of the $\MRprop_p$ constant for the glueing of finitely many operators in terms of these geometrical parameters.

\begin{thm}
\label{thm:finite_glueing}
Let $\mcl{T} = \cbk*{\tau_i}_{0 \leq i \leq N}$ be any finite subdivision of $\ol{I}$. Let us be given a family $\cbk*{A_i}_{1 \leq i\leq N}$ of nonautonomous operators such that $A_i \in \MRprop_p(\tau_{i-1},\tau_i)$. \\

Then the operator $A$ defined by
\begin{equation*}
A(t):= \left\{
\begin{array}{lr}
A_1(t), & \text{ if } \tau_0 \leq t < \tau_1\\
A_2(t), & \text{ if } \tau_1 \leq t < \tau_2\\
\vdots &\\
A_N(t), & \text{ if } \tau_{N-1} \leq t \leq \tau_N
\end{array}
\right.
\end{equation*}
belongs to $\MRprop_p(a,b)$ and we have

\begin{equation}
\label{eq:glueing_finite_A}
\sbk*{A}_{\MRprop_p(a,b)} \leq \sum_{i = 1}^N \frac{\tau_j - \tau_0}{\tau_j - \tau_{j-1}}\pr*{\prod_{j = i+1}^N H_j(p,\mcl{T}, \norm*{A_j}_{\infty}, \sbk*{A_j})}\sbk*{A_i}_{\MRprop_p(\tau_{i-1}, \tau_i)}
\end{equation}
\end{thm}

\begin{proof}
We show the estimate by induction on $N$.
\begin{itemize}

\item[$\star$] For $N = 1$, the result is trivial.

\item[$\star$] Fix some $N \geq 2$ and assume that the result holds for any subdivision of size $N-1$, let us show that the result holds for any subdivision of size $N$. \\

Let $a, b \in \R$, $a < b$, $\mcl{T}:= \cbk*{\tau_i}_{0 \leq i \leq N}$ be a subdivision of size $N$ and let $\cbk*{A_i}_{1 \leq i \leq N}$ be a family of maximally regular nonautonomous operators on each $[\tau_{i-1}, \tau_i]$ respectively. Denote by $A$ the resulting glued-up operator on $[a,b]$. \\

Denote $B_1:= A|_{[\tau_0, \tau_{N-1}]}$ and $B_2:= A_N$. By our induction hypothesis, we are able to apply \autoref{lem:glueing} on $B_1$ and $B_2$. We obtain:
\begin{align*}
\sbk*{A}_{\MRprop_p(a,b)} &\leq \pr*{\kappa_{1, N} + \kappa^{\frac{1}{p}}_{1, N}\kappa^{\frac{1}{q}}_{2, N}Q_p\pr*{\tau_{N-1} - \tau_N, \norm*{B_2}_{\infty}, \sbk*{B_2}}}\sbk*{B_1}_{\MRprop_p(\tau_0, \tau_{N-1})} + \kappa_{2,N}\sbk*{B_2}_{\MRprop_p(\tau_{N-1}, \tau_N)}\\
& = H_N(p,\mcl{T}, \norm*{B_2}_{\infty}, \sbk*{B_2})\sbk*{B_1}_{\MRprop_p(\tau_0, \tau_{N-1})} + \kappa_{2,N}\sbk*{B_2}_{\MRprop_p(\tau_{N}, \tau_N)}
\end{align*}
Let us then apply our induction hypothesis on the subdivisions $\mcl{T}' = \cbk*{\tau_i}_{0 \leq i \leq N-1}$ and the corresponding family of operators $\cbk*{A_i}_{1 \leq i \leq N-1}$ glueing-up to $B_1$ and each of size $N-1$. For simplicity denote $H_j:= H_j(p, \mcl{T}, \norm*{A_j}_\infty, \sbk*{A_j}_{\MRprop_p(\tau_j - \tau_{j-1})})$. We infer:
\begin{align*}
\sbk*{A}_{\MRprop_p(a,b)} & \leq H_N\sum_{i = 1}^{N-1} \pr*{\prod_{j = i+1}^{N-1} H_j}\kappa_{2,i}\sbk*{A_i}_{\MRprop_p(\tau_{i-1}, \tau_i)} + \kappa_{2,N}\sbk*{A_N}_{\MRprop_p(\tau_{N-1}, \tau_N)}\\
& = \sum_{i = 1}^{N-1} \pr*{\prod_{j = i+1}^{N} H_j}\kappa_{2,i}\sbk*{A_i}_{\MRprop_p(\tau_{i-1}, \tau_i)} + \kappa_{2,N}\sbk*{A_N}_{\MRprop_p(\tau_{N-1}, \tau_N)}\\
& = \sum_{i = 1}^N \pr*{\prod_{j = i+1}^{N} H_j}\kappa_{2,i}\sbk*{A_i}_{\MRprop_p(\tau_{i-1}, \tau_i)}
\end{align*}
This completes our proof. 

\end{itemize}
\end{proof}

%
%

\section{Proofs of the main estimates and examples}

In this section we prove our main estimation results. The general estimate  \autoref{thm:relative_continuity} is  derived in \Cref{subsection:4.1}. It is then used in \Cref{subsection:4.2} to show the quantitative version \autoref{thm:A_MR_log_estimate} for operators in $RC^{\alpha, \beta}(I)$. Finally \Cref{subsection:4.3} contains some examples.

\label{section:4}
%

\subsection{Estimation for general relatively continuous operators}
\label{subsection:4.1}

The idea for the study of the maximal regularity of relatively continuous operators, is to subdivide the time interval $[a,b]$ and consider the operator on each subinterval $I_i \subset [a,b]$ as the sum of an autonomous maximal operator $A_i$ and a regular perturbation $B$ arising from the relative continuity. \\

We start with a somewhat precise version of the Besicovitch covering theorem in dimension $1$ (found e.g. in \cite[Theorem 2.18]{ambrosio2000functions}).

\begin{lem}
\label{lem:besicovitch_dim1}
Let $\rho: [a,b] \rightarrow (0,+\infty)$ be a positive function, then there exists $N = N(\rho) \in \N$, a set of center points $\mcl{C}:= \mcl{C}(\rho) = \cbk*{t_i}_{1\leq i \leq N} \subset [a,b]$ and a subdivision of $[a,b]$, $\mcl{T}:= \mcl{T}(\rho) = \cbk*{\tau_i}_{0 \leq i \leq N}$ such that
\begin{equation}
\label{eq:subdivision}
a=\tau_0 \leq t_1 < \tau_1 < t_2 < ... \tau_{N-1} < t_N \leq \tau_N = b
\end{equation}
such that for all $i \in \cbk*{1, ..., N}$, 
\begin{equation}
\label{eq:max_dist_ti}
|t_i-\tau_i|, |t_i-\tau_{i-1}| \leq \rho(t_i)
\end{equation}
\begin{equation}
\label{eq:tau_separation}
\rho(t_i) \leq |\tau_i - \tau_{i-1}| \leq 2 \rho(t_i)
\end{equation}
Hence we can write
\begin{equation*}
[a,b] = \bigcup_{i = 1}^{N} [\tau_{i-1}, \tau_i]
\end{equation*}
\end{lem}

\begin{proof}
We have $[a,b] \subset \bigcup_{t \in [a,b]} I_t$, $I_t:= (t-\rho(t),t + \rho(t))$. By compactness there exists a finite family of intervals $\mcl{F} \subset \cbk*{I_t}_{t \in [a,b]}$ covering $[a,b]$.\\

\begin{itemize}
\item[$\star$] Claim: there exists a subcover $\ti{\mcl{F}} \subset \mcl{F}$ which satisfies both of the following nonoverlapping properties: 

\begin{enumerate}[label = ($\mathbf{P\arabic*}$), leftmargin = 8em]
\item \label{i:p1} For any $I,J \in \ti{\mcl{F}}$, $I\subset J \Rightarrow I = J$. 
\item \label{i:p2} For any $I, J, K \in \ti{\mcl{F}}$, $I \cap J \cap K \neq \emptyset \Rightarrow I = J$ or $J = K$ or $K = I$. 
\end{enumerate}

We show the claim by induction on the size $N$ of $\mcl{F}$.
\begin{description}
\item[Base:] If $N = 1$, the result is trivial.
\item[Induction:] Assume the claim holds for some $N \geq 1$. Let $\mcl{F}$ be a covering family of size $N + 1$. If $\mcl{F}$ already satisfies \ref{i:p1} and \ref{i:p2} we are done. Otherwise, either \ref{i:p1} or \ref{i:p2} fail to be satisfied.

\begin{itemize}
\item[Case 1:] Assume that \ref{i:p1} fails for $\mcl{F}$. There exist $I,J \in \mcl{F}$ such that $I \subset J$ and $I \neq J$. Observe then that $\mcl{F} \setminus \cbk*{I}$ still covers $[a,b]$ and is of size $N$. 

\item[Case 2:] Assume now that \ref{i:p2} fails for $\mcl{F}$. There exist $I,J,K \in \mcl{F}$, all distinct such that $I\cap J \cap K \neq \emptyset$. Write $I = (a_I, b_I)$, $J = (a_J, b_J)$ and $K = (a_K, b_K)$ and assume without loss of generality that $a_I \leq a_J \leq a_K$.
\begin{itemize}
\item[$\star$] If $b_K \leq b_I$ or $b_K \leq b_J$ then we conclude that $K \subset I$ or $K \subset J$ hence $\mcl{F}\setminus \cbk*{K}$ covers $[a,b]$ and is of size $N$.

\item[$\star$] Otherwise if $b_K > b_I$ and $b_K > b_J$, then since $I \cap K \neq \emptyset$, $a_K < b_I$. Thus we have $I\cup K = (a_I, b_K)$ and so $J \subset I \cup K$. Therefore $\mcl{F}\setminus\cbk*{J}$ covers $[a,b]$ and is of size $N$.
\end{itemize}
\end{itemize}

In any case we can cover $[a,b]$ with a subfamily $\ti{\mcl{F}} \subset \mcl{F}$ of size $N$. The induction hypothesis allows to conclude.
\end{description}

\item[$\star$] We still denote by $\mcl{F} = \cbk*{I_{t_i}}_{1 \leq i \leq N}$ some finite subcover of $[a,b]$ satisfying $\mathbf{(P1)}$ and $\mathbf{(P2)}$. Let us assume furthermore that $(t_i)_{1 \leq i \leq N}$ are ordered in an increasing fashion, i.e. $a\leq t_1 < t_2 < ... < t_N \leq b$ and define $\tau_0 = a$, $\tau_N = b$ and $\tau_i:= \frac{t_{i+1} + t_i}{2} + \frac{\rho(t_i) - \rho(t_{i + 1})}{2}$ for $i \in \cbk*{0,..., N-1}$. From the properties \ref{i:p1} and \ref{i:p2}, we infer for any $i \in \cbk*{1, ..., N-1}$ that,

\begin{equation*}
\rho(t_i) - \rho(t_{i+1}) < t_{i+1} - t_i < \rho(t_i) + \rho(t_{i + 1})
\end{equation*}
and
\begin{equation*}
t_{i+1} - \rho(t_{i+1}) > t_{i-1} + \rho(t_{i - 1})
\end{equation*}
From these we straightforwardly infer \eqref{eq:subdivision} and \eqref{eq:tau_separation}. 
\end{itemize}

%
\end{proof} 

Let us now state and prove an explicit version of the result of Arendt (found in \cite[Theorem 2.7, Theorem 2.11]{arendt2007lp-maximal}). First let us recall the quantities of interest and the main assumption of the theorem. 

\begin{enumerate}[label = \textbf{(A\arabic*)}, leftmargin = 40pt]
\item \label{item:A1_s5}
$\left\{
\begin{varwidth}{\textwidth}
\begin{enumerate}[label = (\alph*) ]
\item \label{item:a_s5} \textbf{Relative continuity:} $A \in RC(I;(D;X))$.
\item \label{item:b_s5} \textbf{Decomposition:} There exist $A_0: I \lra \mcl{L}(D,X)$ and $B: I \times I \lra \mcl{L}(D,X)$ such that 
\begin{enumerate}[label = (\alph{enumii}.\arabic*), leftmargin = 40pt]
\item \label{item:b1_s5} For all $t, s \in I$, $A(s) = A_0(t) + B(t,s)$.
\item \label{item:b2_s5} For all $t \in I$, $A_0(t) \in \mrprop_p(\R_+)$.
\item \label{item:b3_s5} There exists a range $r_A = (\delta_A, \eta_A)$ such that for all $x \in D$,
\begin{equation}
\forall t, s\in I, \quad |t-s| \leq \rho_A(t) \quad \Longrightarrow \quad 
\norm*{B(t,s) x}_X \leq \epsilon_0(t) \norm*{x}_D + \mu_A(t) \norm*{x}_X
\end{equation}
where $\rho_A:= \delta_A \circ \epsilon_0$, $\mu_A:= \eta_A \circ \epsilon_0$ and $\epsilon_0$ is defined through \eqref{eq:def_KA_MA}-\eqref{eq:def_epsA}. 
\end{enumerate}
\end{enumerate}
\end{varwidth}
\right.$
\end{enumerate}

Recall also the auxiliary quantities $K_0$ and $M_0$ for $A_0$ defined in \eqref{eq:def_KA_MA}.

\begin{thm}
	Assume that $A$ satisfies \ref{item:A1_s5}. Then $A \in \MRprop_p(I)$ and there exists a subdivision $\mcl{T}:= \cbk*{\tau_i}_{0 \leq i \leq N}$ of $\ol{I}$ and center points $\mcl{C}:= \cbk*{t_i}_{1 \leq i \leq N}$ such that 
	\begin{gather}
		a = \tau_0 < t_1 < \tau_1 < ... < t_N < \tau_N = b\\
		\rho_A(t_i) \leq \tau_i - \tau_{i-1} \leq 2 \rho_A(t_i), \qquad \forall i \in \cbk*{1,..., N}
	\end{gather}
	and there exists $K:= K(p, \mcl{T}, \mcl{C}, K_0, M_0, \norm*{A}_\infty) \geq 0$ such that
	\begin{equation}
	\label{eq:bound_relatively_continuous2}
	\sbk*{A}_{\MRprop_p(I)} \leq K
	\end{equation}
\end{thm}

\begin{proof}
Apply \autoref{lem:besicovitch_dim1} with $\rho_A = \delta_A \circ \epsilon_0$ on $[a,b]$. This yields the corresponding $N \in \N$, $\mcl{C}:= (t_i)_{1 \leq i \leq N}$ and $\mcl{T}:= (\tau_i)_{0 \leq i \leq N}$.\\

Define $A_i:= A|_{\tau_{i-1}, \tau_i}$ for each $i \in \cbk*{1,..., N}$. By the decomposition assumption \ref{item:b_s5}, we have $A_i(s) = A_0(t_i) + B(t_i, s)$ and by \eqref{eq:max_dist_ti}, $s \longmapsto B(t_i, s)$ is a regular perturbation of $A_0(t_i)$ on the interval $[\tau_{i-1}, \tau_i]$. Hence by \autoref{thm:perturbation} we have $A_i \in \MRprop_p(\tau_{i-1}, \tau_i)$ and,

\begin{equation*}
\sbk*{A_i}_{\MRprop_p(\tau_{i-1}, \tau_i)} \leq G(\tau_i - \tau_{i-1}, \mu_A(t_i), M_0(t_i), K_0(t_i))
\end{equation*}

Finally we can apply \autoref{thm:finite_glueing} to $A$ with the subdivision $\mcl{T}$ to obtain \eqref{eq:bound_relatively_continuous2}.  

\end{proof}

\subsection{Estimation for Hölder-type relatively continuous operators}
\label{subsection:4.2}

In order to estimate the $L^p$ regularity constant more explicitly, we use an integrability condition on the maps $K_0$ and $M_0$ stated in \ref{item:A2} \ref{item:b_A2}: $(K_0 + 1)(M_0 + 1) \in L^{\alpha r}(I)$. This implies in fact that $\frac{1}{\rho_A} \in L^r(I)$. \\

Under this condition, the Besicovitch covering theorem can be explicited with concrete constants with the use of Techebyshev inequalities.

\begin{prop}
\label{prop:explicit_bound_N}
Let $\rho: I \lra \R_+$ be measurable and such that $\frac{1}{\rho} \in L^r(I)$ for some $r \in (1, +\infty]$. Then there exist $(N, \cbk*{t_i}_{1 \leq i\leq N}, \cbk*{\tau_i}_{0 \leq i \leq N})$ satisfying \eqref{eq:subdivision} and such that 
\begin{equation}
\label{eq:explicit_bound_N}
\ceil*{\frac{1}{2}|I|\norm*{\frac{1}{\rho}}^{r^*}_{L^r(a,b)}}\leq N \leq \floor*{|I|\norm*{\frac{1}{\rho}}^{r^*}_{L^r(I)}}\vee 1
\end{equation}
where $r^*$ is such that $\frac{1}{r} + \frac{1}{r^*} = 1$.
\end{prop}

\begin{proof}
Denote for $\epsilon > 0$, $E_\epsilon:= \cbk*{t \in \ol{I} \: \ \rho(t) > \epsilon}$. From Markov's inequality we have for any $\epsilon >0$, 
\begin{equation}
\label{eq:markov_1/rho}
\snorm*{\ol{I}\setminus E_\epsilon} \leq \epsilon^r\norm*{\frac{1}{\rho}}_{L^r(I)}^{r}
\end{equation}
Hence there exists $\epsilon^*:= \frac{|I|^\frac{1}{r}}{\norm*{\frac{1}{\rho}}_{L^r}} > 0$ such that $E_{\epsilon} \neq \emptyset$ for any $0 < \epsilon < \epsilon^*$. If $\epsilon^* > |I|$, then we may take $N = 1$, hence \eqref{eq:explicit_bound_N} obviously holds. From now on we assume that $\epsilon^* \leq |I|$. \\

Fix some $0 <\alpha < 1$, and note that if there exists some $0< \epsilon_0 < \epsilon^*$ such that
\begin{equation}
\label{eq:condition_eps}
|E_{\epsilon_0}| \geq |I| - \alpha \epsilon_0
\end{equation}
then for any $t \in \ol{I}$, we have $\dist(t,E_{\epsilon_0}) \leq \alpha \epsilon_0$ otherwise we could fit a ball of diameter strictly bigger than $\alpha \epsilon_0$ outside $E_{\epsilon_0}$ which would contradict \eqref{eq:condition_eps}. Therefore we have

\begin{equation*}
\ol{I} \subset \bigcup_{t \in E_{\epsilon_0}} B(t,\epsilon_0)
\end{equation*}
Compactness of $\ol{I}$ allows to extract a finite subcover, which we may assume satisfies \ref{i:p1}-\ref{i:p2}, yielding \eqref{eq:subdivision} with $\epsilon_0 \leq |\tau_i - \tau_{i-1}| \leq 2 \epsilon_0$ for all $1 \leq i \leq N$. Therefore we can bound $N$ from above and below:
\begin{equation}
\label{eq:encadrement_N}
\ceil*{\frac{|I|}{2 \epsilon_0}} \leq N \leq  \floor*{\frac{|I|}{\epsilon_0}}
\end{equation}
Now from \eqref{eq:markov_1/rho}, condition \eqref{eq:condition_eps} is satisfied for $\epsilon_0$ satisfying 
\begin{equation*}
\epsilon_0^r\norm*{\frac{1}{\rho}}^r_{L^r(I)} = \alpha \epsilon_0
\end{equation*}
After some rearranging we find that $\epsilon_0:= \frac{\alpha^{r^* - 1}}{\norm*{\frac{1}{\rho}}_{L^r(I)}^{r^*}} < \epsilon^*$ because we assumed $\epsilon^* \geq |I|$. Denoting $\beta = \pr*{\frac{1}{\alpha}}^{\frac{1}{r^*-1}} \in (1, +\infty)$, we find from \eqref{eq:encadrement_N}

\begin{equation}
\label{eq:explicit_bound_N_beta}
\ceil*{\frac{\beta}{2}|I|\norm*{\frac{1}{\rho}}^{r^*}_{L^r(I)}}\leq N \leq \floor*{\beta |I|\norm*{\frac{1}{\rho}}^{r^*}_{L^r(I)}}
\end{equation}

Letting $\beta$ be close to $1$ in \eqref{eq:explicit_bound_N_beta}, we obtain \eqref{eq:explicit_bound_N}. 

\end{proof}

Assume that $A$ satisfies \ref{item:A1} and \ref{item:A2}. Recall that then $A \in RC^{\alpha, \beta}(I)$: it has a specific range of relative continuity $r_A = (\delta_A, \eta_A)$ such that there exists $m_\delta \geq 0$, $m_\delta \in (0,1]$ and

\begin{equation*}
\forall \epsilon > 0, \quad \delta_A(\epsilon) = m_\delta \epsilon^\alpha
\end{equation*}
\begin{equation*}
\forall \epsilon > 0, \quad \eta_A(\epsilon) \leq m_\eta \epsilon^{-\beta}
\end{equation*}

The integrability condition \ref{item:A2} \ref{item:b_A2}, lets us define $\Gamma:= \norm*{(K_0 + 1)(M_0 + 1)}_{L^{\alpha r}(I)}^\alpha$.  

\begin{thm}
Assume that $A$ satisfies \ref{item:A1} and \ref{item:A2}. Then there exists a constant $C = C(|I|, \norm*{A}_\infty, p, r, \alpha, \beta)$ such that,
\begin{equation}\label{eq:A_MR_log_estimate'}
\log\pr*{\sbk*{A}_{\MRprop_p(I)}} \leq C\pr*{\pr*{\frac{\Gamma}{m_\delta}}^{r^*}
\left\{
\begin{array}{c}
1\\
+\\
m_\eta m_\delta^{s_2(\alpha, \beta, r)} \Gamma^{s_1(\alpha, \beta, r)}\\
+ \\
\displaystyle \log\pr*{\frac{\Gamma}{m_\delta}}
\end{array}
\right\} + m_\eta + 1}
\end{equation}
where $s_1(\alpha, \beta, r):= r^*\pr*{\frac{\beta + 1}{\alpha}-1}$ and $s_2(\alpha, \beta, r):= \frac{\beta + 1}{\alpha + 1} - \pr*{\frac{\alpha + \frac{1}{r}}{\alpha + 1}}s_1(\alpha, \beta, r)$.
\end{thm}

This result is useful for applications, it allows to express an explicit bound of the regularity constant in terms of the parameters of relative continuity of the nonautonomous operator $A$. This estimate is then used to infer a global existence result for quasilinear equations, stated in the next section.

\begin{proof}
  The proof, which is quite technical, can be found in \cite{belin2024wp}.
\end{proof}

\subsection{Examples}
\label{subsection:4.3}

We here give two simple but useful examples of operators in the class $RC^{\alpha, \beta}(I)$ for which we can apply our explicit estimate in \autoref{thm:A_MR_log_estimate}. \\

Throughout the development of these examples we fix $\ol{A} \in \mrprop_p(\R_+)$ a given maximally regular operator on $\R_+$ and an additional operator $B \in \mcl{L}(D';X)$ where $D \cinj D' \cinj X$ is an intermediate space between $D$ and $X$, that is we assume there exists $L \geq 0$, $\theta \in (0,1)$ such that for any $x \in D$ the following interpolation inequality holds,
\begin{equation*}
\norm*{Bx}_{X} \leq L\norm*{x}_D^\theta \norm*{x}_X^{1-\theta}
\end{equation*}
The reader may think of $\ol{A}$ as a well-known, autonomous and maximally regular operator (e.g. $A = -\Delta$ on $L^2(\R^n)$, $A = -\div_x\pr*{a(x)\nabla_x \cdot}$ with $a, a^{-1} \in L^\infty(\R^n; \mcl{S}^{++}(\R))$) and $B$ as representing an operator with "low-order coefficients" (e.g. Gagliardo-Nirenberg interpolation inequalities). 

\begin{ex}
\label{ex:log_estimate_lambdaA}
Fix some $\lambda \in  C^{0,\frac{1}{\alpha}}(\ol{I}; \R_+)$ such that $\frac{1}{\lambda^{2}} \in L^{\alpha r}(I)$ for some $r > 1$ and define for $t \in I$,
\begin{equation*}
A(t):= \lambda(t)\ol{A}
\end{equation*}
Then $A \in C^{0,\frac{1}{\alpha}}(I;\mcl{L}(D;X)) \subset RC^{\alpha, \beta}(I)$ with $m_\delta = \pr*{1 + \norm*{A}_{\mcl{L}(D;X)}\sbk*{\lambda}_{\frac{1}{\alpha}}}^{-\alpha}$, where ${\displaystyle \sbk*{\lambda}_{\frac{1}{\alpha}}:= \sup_{t\neq s \in I} \frac{|\lambda(t) - \lambda(s)|}{|t-s|^{\frac{1}{\alpha}}}}$ is the $\frac{1}{\alpha}$-Hölder semi-norm of $\lambda$. \\

The operator $A$ satisfies assumption \ref{item:A1} with $A_0 = A$ and $B = 0$. Now observe that 
\begin{align*}
K_0(t) &:= \sbk*{A(t)}_{\mrprop_p(I)} \\
& = \frac{1}{\lambda(t)}\sbk*{\ol{A}}_{\mrprop_p\pr*{\frac{I}{\lambda(t)}}} \\
&\leq \frac{|I| + \lambda(t)}{\lambda(t)^2}\sbk*{\ol{A}}_{\mrprop_p(\R_+)}\\
& \leq C(\ol{A}, |I|, \norm*{\lambda}_\infty)\frac{1}{\lambda(t)^2}
\end{align*}
first by time rescaling invariance (see \eqref{eq:mr_const_Aphi_affine} in \autoref{rem:change_var_applications}) then by \autoref{prop:mr_injections}. Also $M_0(t):= \sup_{\tau \geq 0} \norm*{e^{-\tau \ol{A}}}_X$, therefore
\begin{equation*}
\Gamma \leq C(\ol{A}, |I|, \norm*{\lambda}_\infty) \norm*{\frac{1}{\lambda^{2}}}^\alpha_{L^{\alpha r}(I)}
\end{equation*}
Therefore from \autoref{thm:A_MR_log_estimate} \eqref{eq:A_MR_log_estimate} we find that there exits $C(\ol{A}, |I|, \norm*{\lambda}_\infty, p, \alpha, r) \geq 0$ such that 
\begin{equation}
\label{eq:log_estimate_lambdaA}
\log\pr*{\sbk*{A}_{\MRprop_p(I)}} \leq C\pr*{1 + \sbk*{\lambda}_{\frac{1}{\alpha}}}^{\alpha r^*}\norm*{\frac{1}{\lambda^{2}}}^{\alpha r^*}_{L^{\alpha r}}\pr*{ 1 + \log\pr*{\pr*{1 + \sbk*{\lambda}_{\frac{1}{\alpha}}}\norm*{\frac{1}{\lambda^{2}}}_{L^{\alpha r}}}}
\end{equation}
We observe from estimate \eqref{eq:log_estimate_lambdaA} that the behavior of the regularity constant of $A$ can be expressed in terms of the strength of the oscillations of $\lambda$ represented by $\sbk*{\lambda}_{\frac{1}{\alpha}}$ and of the term $\norm*{\frac{1}{\lambda^{2\alpha}}}_{L^r}$ representing a measure of the degeneracy of the operator. 
\end{ex}

\begin{ex}
\label{ex:log_estimate_lambdaA_gammaB}
We now consider a more general operator of the form 
\begin{equation*}
A(t) = \lambda(t)\ol{A} + \gamma(t)B
\end{equation*}
where as before $\lambda \in C^{0,\frac{1}{\alpha}}(\ol{I}; \R_+)$ and $\gamma \in L^\infty(I; \R)$. As above we show how to obtain a log-estimate of $\sbk*{A}_{\MRprop_p(I)}$ using \autoref{thm:A_MR_log_estimate}. 

\begin{description}
\item[Relative continuity:] $A \in RC^{\alpha, \frac{\theta}{1-\theta}}(I)$.\\

Let us compute a range of relative continuity for $A(\cdot)$. Let $t,s \in I$, $t \neq s$ and $x \in D$, we have, 

\begin{align*}
\norm*{A(t)x - A(s)x}_X &\leq |\lambda(t) - \lambda(s)|\norm*{\ol{A}}_{\mcl{L}(D;X)}\norm*{x}_D + |\gamma(t) - \gamma(s)|\norm*{Bx}_X\\
& \leq \sbk*{\lambda}_{\frac{1}{\alpha}}\norm*{\ol{A}}_{\mcl{L}(D;X)}|t-s|^{\frac{1}{\alpha}}\norm*{x}_D + 2L\norm*{\gamma}_\infty \norm*{x}_D^{\theta} \norm*{x}_X^{1-\theta}
\end{align*}
Now for any $h > 0$, from Young's inequality with $p = \frac{1}{\theta}$, $q = \frac{1}{1- \theta}$, we have $a^\theta b^{1- \theta} \leq \theta h^{\frac{1}{\theta}} a + (1-\theta)h^{-\frac{1}{1-\theta}}b$. Therefore, 

\begin{equation*}
\norm*{A(t)x - A(s)x}_X \leq \pr*{\sbk*{\lambda}_{\frac{1}{\alpha}}\norm*{\ol{A}}_{\mcl{L}(D;X)}|t-s|^{\frac{1}{\alpha}} + 2\theta L\norm*{\gamma}_\infty  h^{\frac{1}{\theta}}}\norm*{x}_D + 2(1-\theta) L\norm*{\gamma}_\infty h^{-\frac{1}{1-\theta}} \norm*{x}_X
\end{equation*}

Fix $r > 0$ and take $h = \pr*{r|t-s|}^{\frac{\theta}{\alpha}}$ to find, 

\begin{equation*}
\norm*{A(t)x - A(s)x}_X \leq \pr*{\sbk*{\lambda}_{\frac{1}{\alpha}}\norm*{\ol{A}}_{\mcl{L}(D;X)} + 2\theta L\norm*{\gamma}_\infty r^{\frac{1}{\alpha}}}|t-s|^{\frac{1}{\alpha}}\norm*{x}_D + 2(1-\theta)L\norm*{\gamma}_\infty \pr*{r|t-s|}^{-\frac{\theta}{\alpha \pr*{1-\theta}}} \norm*{x}_X
\end{equation*}

Now choose $r:= \pr*{\frac{\sbk*{\lambda}_{\frac{1}{\alpha}}\norm*{\ol{A}}}{2 \theta L \norm*{\gamma}_\infty}}^\alpha$ to find, 

\begin{align*}
\norm*{A(t)x - A(s)x}_X &\leq 2\sbk*{\lambda}_{\frac{1}{\alpha}}\norm*{\ol{A}}_{\mcl{L}(D;X)}|t-s|^{\frac{1}{\alpha}}\norm*{x}_D + 2(1-\theta)L\norm*{\gamma}_\infty\pr*{\frac{2 \theta L \norm*{\gamma}_\infty}{\sbk*{\lambda}_{\frac{1}{\alpha}}\norm*{\ol{A}}}}^{\frac{\theta}{\pr*{1-\theta}}} |t-s|^{-\frac{\theta}{\alpha \pr*{1-\theta}}} \norm*{x}_X\\
& = C_\lambda |t-s|^{\frac{1}{\alpha}}\norm*{x}_D + (2\theta)^{\frac{\theta}{1-\theta}}(1-\theta)C_\gamma \pr*{\frac{C_\gamma}{C_\lambda}}^{\frac{\theta}{\pr*{1-\theta}}} |t-s|^{-\frac{\theta}{\alpha \pr*{1-\theta}}}\norm*{x}_X
\end{align*}
where $C_\lambda:= 2\sbk*{\lambda}_{\frac{1}{\alpha}}\norm*{\ol{A}}_{\mcl{L}(D;X)}$ and $C_\gamma:= 2 L\norm*{\gamma}_\infty$. Observe then that for any $\epsilon > 0$, if $|t-s| \leq C_\lambda^{-\alpha} \epsilon^\alpha$, then we find

\begin{align*}
\norm*{A(t)x - A(s)x}_X &\leq \epsilon\norm*{x}_D + C_\theta C_\gamma\pr*{\frac{C_\gamma}{C_\lambda}}^{\frac{\theta}{1-\theta}} C_\lambda^{\frac{\theta}{1-\theta}}\epsilon^{-\frac{\theta}{1-\theta}}\norm*{x}_X\\
&\leq \epsilon\norm*{x}_D + C_\theta C_\gamma^{\frac{1}{1-\theta}}\epsilon^{-\frac{\theta}{1-\theta}}\norm*{x}_X
\end{align*}
A range of relative continuity of $A$ can be
\begin{equation*}
\ol{\delta}(\epsilon):= \pr*{C_\lambda + 1}^{-\alpha} \epsilon^{\alpha}
\end{equation*}
\begin{equation*}
\ol{\eta}(\epsilon):= C_\theta C_\gamma^{\frac{1}{1-\theta}}\epsilon^{-\frac{\theta}{1-\theta}}
\end{equation*}
and therefore $A \in RC^{\alpha, \frac{\theta}{1-\theta}}(I)$.

\item[Maximal regularity:] $A \in \MRprop_p(I)$.\\
Here $A$ satisfies assumption \ref{item:A1} with $A_0: t\mapsto \lambda(t)\ol{A}$ and $B: (t,s) \mapsto \pr*{\lambda(s)-\lambda(t)}\ol{A} + \gamma(s)B$. Now we have for $x \in D$, $\epsilon > 0$

\begin{align*}
\norm*{B(t,s)x}_X &\leq \norm*{A(s)x - A(t)x}_X + \norm*{\gamma}_\infty\norm*{Bx}_X\\
& \leq \norm*{A(s)x - A(t)x}_X + \frac{\epsilon}{2} \norm*{x}_D + C_\theta C_\gamma \epsilon^{-\frac{\theta}{1-\theta}}\norm*{x}_X
\end{align*}

Hence upon choosing $\delta_A(\epsilon):= \ol{\delta}(\frac{\epsilon}{2})$ and $\eta_A(\epsilon):= \ol{\eta}(\frac{\epsilon}{2}) + C_\theta C_\gamma \epsilon^{-\frac{\theta}{1- \theta}}$ we find that $(\delta_A, \eta_A)$ is still a range of relative continuity for $A$, compatible in $RC^{\alpha, \frac{\theta}{1-\theta}}(I)$ and that $B$ satisfies \ref{item:A1}\ref{item:b3}. \\

\item[Logarithmic estimate:] $\log\pr*{\sbk*{A}_{\MRprop_p(I)}}$. \\
As in the continuous case, we find that
\begin{equation*}
\Gamma \leq C(\ol{A}, |I|, \norm*{\lambda}_\infty) \norm*{\frac{1}{\lambda^2}}^\alpha_{L^{\alpha r}(I)} < +\infty
\end{equation*}

Hence $A$ satisfies \ref{item:A2} for $\alpha$ and $\beta = \frac{\theta}{1-\theta}$. We also observe from the definitions of $\delta_A$ and $\eta_A$ that 
\begin{equation*}
m_\delta \propto (1 + \sbk*{\lambda}_{\frac{1}{\alpha}})^{-\alpha}
\end{equation*} 
\begin{equation*}
m_\eta \propto \norm*{\gamma}_\infty^{\frac{1}{1-\theta}}
\end{equation*}
where the proportionality factors only depend on $\theta$, $L$ and $\norm*{\ol{A}}_{\mcl{L}(D,X)}$. Also see that $s_1(\alpha, \frac{\theta}{1- \theta}) = r^*\pr*{\frac{1}{\alpha(1-\theta)} - 1}$ and $s_2(\alpha, \frac{\theta}{1-\theta}) = \frac{1}{(1-\theta)(\alpha + 1)} - r^*\pr*{\frac{\alpha + \frac{1}{r}}{\alpha + 1}}\pr*{\frac{1}{(1-\theta)\alpha} - 1}$. Hence we can apply \autoref{thm:A_MR_log_estimate}, to find 

\begin{equation}
\label{eq:log_estimate_lambdaA_gammaB}
\log\pr*{\sbk*{A(\cdot)}_{\MRprop_p(I)}} \leq 
C\pr*{\pr*{1 + \sbk*{\lambda}_{\frac{1}{\alpha}}}^{\alpha r^*}\norm*{\frac{1}{\lambda^2}}^{\alpha r^*}_{L^{\alpha r}}
\left\{
\begin{array}{ccc}
1 \\
+ \\
\norm*{\gamma}_\infty^{\frac{1}{1-\theta}}\pr*{1 + \sbk*{\lambda}_{\frac{1}{\alpha}}}^{\alpha s_2}\norm*{\frac{1}{\lambda^2}}_{L^{\alpha r}}^{\alpha s_1}\\
+\\
\log\pr*{\pr*{1 + \sbk*{\lambda}_{\frac{1}{\alpha}}}\norm*{\frac{1}{\lambda^2}}_{L^{\alpha r}}}\\
\end{array}
\right\}
+\norm*{\gamma}_\infty^{\frac{1}{1-\theta}} + 1}
\end{equation}
where $C:= C(|I|, \norm*{\lambda}_\infty, p, \alpha, \theta, r)$. \\

As for the continuous case described in \autoref{ex:log_estimate_lambdaA}, the behavior of the regularity constant of $A$ in \eqref{eq:log_estimate_lambdaA_gammaB} is expressed in terms of $\sbk*{\lambda}_{\frac{1}{\alpha}}$ representing the oscillations of the diffusion coefficient $\lambda$ and in terms of $\norm*{\frac{1}{\lambda^{2\alpha}}}_{L^r}$ representing a measure of degeneracy of $\lambda$. Extra terms in $\norm*{\gamma}_\infty$ appear which account for the perturbation $\gamma(t)B$. Remark that if $\theta$ is close to $0$, then the log-estimate becomes linear with respect to $\norm*{\gamma}_\infty$ which matches the asymptotic behavior of $\sbk*{A + \gamma}_{\MRprop_p(I)}$ as $\gamma \ra +\infty$, $\gamma \in \R$ (see \autoref{prop:perturbation_dilation}). If $\gamma \equiv 0$ we exactly find \eqref{eq:log_estimate_lambdaA} back. 
\end{description}

A regime of interest happens when $\alpha(1-\theta) < 1$, in this case we have $s_1, s_2 < 0$. Therefore the above bound can be simplified as follows: 

\begin{equation*}
\log\pr*{\sbk*{A(\cdot)}_{\MRprop_p(I)}} \leq 
C\pr*{\pr*{1 + \sbk*{\lambda}_{\frac{1}{\alpha}}}^{\alpha r^*}\norm*{\frac{1}{\lambda(\cdot)^{2}}}^{\alpha r^*}_{L^{\alpha r}}
\left\{
\begin{array}{ccc}
1 \\
+ \\
\norm*{\gamma}_\infty^{\frac{1}{1-\theta}}\\
+\\
\log\pr*{\pr*{1 + \sbk*{\lambda}_{\frac{1}{\alpha}}}\norm*{\frac{1}{\lambda(\cdot)^{2}}}_{L^{\alpha r}}}\\
\end{array}
\right\} + 1}
\end{equation*}

\end{ex}

\section{Global existence for quasilinear equation}
\label{section:5}
\subsection{A general existence result}

Let us recall here the criterion for the existence of global solutions on the bounded interval $I$ of the following quasilinear problem:

\begin{equation}
\label{eq:non_local_quasilinear'}
\left\{
\begin{array}{rcl}
\displaystyle \ddt u + \mathbb{A}(u)u &= &\mathbb{F}(u)\\
\\
u(a) &=& x
\end{array}
\right.
\tag{QL}
\end{equation}
where $x \in \Tr^p$.  

\begin{enumerate}[label = (\textbf{E}$_{\Roman*}$), leftmargin = 30pt]
 
	\item \label{item:E_I2} For all $R \geq 0$, define the quantities: 
	
	\begin{equation}
	\gamma(R):= \sup_{\norm*{u}_{\mcl{X}(I)} = R}\sbk*{\mathbb{A}(u)}_{\MRprop_p(I)}
	\end{equation}
	\begin{equation}
	\kappa(R):= \sup_{\norm*{u}_{\mcl{X}(I)} = R}\cbk*{\norm*{\mathbb{A}(u)}_{\infty}  + \norm*{\mathbb{F}(u)}_{L^p(I;X)} + 1}
	\end{equation}
	and assume that for any $L > 0$ the following sublinear growth condition holds,
	\begin{equation}
	\sup_{R_1, R_2 \in [R-L, R + L]}\frac{\gamma(R_1)\kappa(R_2)}{R} \uset{R \ra +\infty}{\lra} 0 
	\end{equation}
	
	\item \label{item:E_II2} We ask the following continuity properties:
	\begin{enumerate}[label = (\roman*)]
	\item $\mathbb{F}: \left\{\begin{array}{rcl}
	\mcl{X}(I) &\lra & L^p(I;X)\\
	u &\longmapsto& \mathbb{F}(u)
	\end{array}\right.$ is weakly continuous.
	\item $\mathbb{A}(\cdot): \left\{\begin{array}{rcl}
	\mcl{X}(I) &\lra &L^\infty(I;\mcl{L}(D,X))\\
	u &\longmapsto& \mathbb{A}(u)
	\end{array}\right.$ is continuous. 

	\end{enumerate}

\end{enumerate}

Let us now prove the existence theorem. 

\begin{thm}
\label{thm:existence_quasilinear'}
Assume that $\mathbb{A}$ and $\mathbb{F}$ satisfy conditions \ref{item:E_I2} and \ref{item:E_II2}. Then for any $x \in \Tr^p$ there exists a unique global strong solution $u \in \MR^p(I)$ of \eqref{eq:non_local_quasilinear'}
\end{thm}

\begin{proof}

\begin{description}

\item[Reduction to $x = 0$: ] Fix a function $w \in \MR^p(I)$ such that $w(a) = x$, and consider the following quasilinear problem: 

\begin{equation}
\label{eq:quasilinear_modif_w}
\left\{
\begin{array}{rcl}
\displaystyle \ddt v + \mathbb{A}(v + w)v &=& \displaystyle \mathbb{F}(v + w) - \ddt w - \mathbb{A}(v + w)w\\
\\
v(a) &=& 0
\end{array}
\right.
\end{equation}

If $v \in \MR^p(I)$ is a solution of \eqref{eq:quasilinear_modif_w}, then $u:= v + w$ is a solution to the original problem \eqref{eq:non_local_quasilinear'} starting from $x$. Let us show that
\begin{equation*}
\mathbb{A}_x(\cdot):= \mathbb{A}(\cdot + w)
\end{equation*}
\begin{equation*}
\mathbb{F}_x(\cdot):= \mathbb{F}(\cdot + w) - \ddt w - \mathbb{A}_x(\cdot)w
\end{equation*}
satisfy assumptions and \ref{item:E_I2} - \ref{item:E_II2}. It is straightforward to see that they satisfy \ref{item:E_II2}. There only needs to show that they satisfy \ref{item:E_I2}. \\

Define as required 
\begin{equation*}
	\gamma_x(R):= \sup_{\norm*{u}_\mcl{X}(I) = R} \sbk*{\mathbb{A}_x(u)}_{\MRprop_p(I)}
\end{equation*}
\begin{equation*}
	\kappa_x(R):= \sup_{\norm*{u}_{\mcl{X}(I)} = R} \cbk*{\norm*{\mathbb{A}_x(u)}_{\infty}  + \norm*{\mathbb{F}_x(u)}_{L^p(I;X)}}
\end{equation*}
For a given $u \in \mcl{X}(I)$, from the triangle inequality we have $\norm*{u + w}_{\mcl{X}(I)} - \norm*{w}_{\mcl{X}(I)} \leq \norm*{u}_{\mcl{X}(I)} \leq \norm*{u + w}_{\mcl{X}(I)}+ \norm*{w}_{\mcl{X}(I)}$. We infer that if $L_w:= \norm*{w}_{\mcl{X}(I)}$, we have for any $R \geq L_w$,
\begin{equation*}
	\gamma_x(R) \leq \sup_{R_1 \in [R-L_w, R+L_w]} \gamma(R_1)
\end{equation*}
In a similar fashion, since there exists $C_w \geq 0$ such that 
\begin{equation*}
	\norm*{\mathbb{F}_x(u)}_{L^p(I;X)} \leq C_w\pr*{\norm*{\mathbb{F}(u + w)}_{L^p(I;X)} + \norm*{\mathbb{A}(u+w)}_\infty + 1}
\end{equation*}
we find that 
\begin{equation*}
	\kappa_x(R) \leq \sup_{R_2 \in [R-L_w, R+L_w]} C_w\kappa(R_2)
\end{equation*}
Let $L \geq 0$ and let $R \geq 0$ and $R_1, R_2 \in [R-L, R + L]$ we have, 
\begin{align*}
	\gamma_x(R_1)\kappa_x(R_2) &\leq \sup_{R_1' \in [R_1-L_w, R_1+L_w]} \gamma(R_1')\sup_{R_2' \in [R_2-L_w, R_2+L_w]} C_w \kappa(R_2')\\
	&\leq C_w \sup_{R_1', R_2' \in [R-L-L_w, R + L + L_w]} \gamma(R_1')\kappa(R_2')
\end{align*}
Now taking the supremum over $R_1$ and $R_2$, we find from \ref{item:E_I'} with $L + L_w$ that

\begin{equation*}
	\sup_{R_1, R_2 \in [R-L, R+L]}\gamma_x(R_1)\kappa_x(R_2) \uset{R \ra +\infty}{\lra} 0
\end{equation*}

\item[Existence: ] The existence of the solution is shown through Schauder's fixed point theorem. Define the following map:
\begin{equation*}
T: \left\{\begin{array}{rcl}
\mcl{X}(I) &\lra& \mcl{X}(I)\\
u &\longmapsto &L_{\mathbb{A}(u)}^{-1} \mathbb{F}(u)
\end{array}\right.
\end{equation*}
and note that any solution $u$ of \eqref{eq:non_local_quasilinear'} must be a fixed point of $T$.

\begin{itemize}

\item[$\star$] The map $T$ is continuous from $\mcl{X}(I)$ equipped with the strong topology to $\MR^p(I)$ equipped with the weak topology as a consequence of \ref{i:weak_weak} of \autoref{prop:elem_result} and the continuity properties of $\mathbb{F}$ and $\mathbb{A}$ in \ref{item:E_II2}. Now since the embedding $\MR^p(I) \cinj \mcl{X}(I)$ is compact, this implies the continuity of $T$ onto the strong topology of $\mcl{X}(I)$. 

\item[$\star$] $T$ is compact, again by the compact embedding $\MR^p(I)\cinj \mcl{X}(I)$. \\

\item[$\star$] Let $u \in \mcl{X}(I)$, and denote $R:= \norm*{u}_{\mcl{X}(I)}$. We have
\begin{align*}
\norm*{Tu}_{\MR^p(I)} &\leq \sbk*{\mathbb{A}(u)}_{\MRprop_p(I)}\norm*{\mathbb{F}(u)}_{L^p(I;X)}\\
& \leq \gamma(R)\kappa(R)\\
& \leq R \pr*{\frac{\gamma(R)\kappa(R)}{R}} 
\end{align*}
Denote by $C > 0$ the continuity constant of the embedding $\MR^p(I) \cinj \mcl{X}(I)$, to find
\begin{equation}
\label{eq:est_Tu}
\norm*{Tu}_{\mcl{X}(I)} \leq CR\pr*{\frac{\gamma(R)\kappa(R)}{R}}
\end{equation}

For any $r > 0$, denote $B_r = \ol{\mathbb{B}}_{\mcl{X}(I)}(0,r)$ the closed ball of $\mcl{X}(I)$ of radius $r \in \R_+$ and centered at $0$.\\

By \ref{item:E_I} there exists $R_0 > 0$ large enough such that for any $R \geq R_0$ we have ${C\frac{\gamma(R)\kappa(R)}{R} \leq \frac{1}{2}}$. \\
By continuity of the map $T$ there exists $R_1 > 0$ large enough such that $T(B_{\frac{R_0}{2}}) \subset B_{R_1}$. We denote $R^*:= R_0 \vee R_1$ and define the quantity 

\begin{equation*}
\alpha:= \sup_{\frac{R_0}{2} \leq R \leq
R^*} C\pr*{\frac{\gamma(R)\kappa(R)}{R}} < +\infty
\end{equation*}

\item[$\star$] We claim that $T(B_{\alpha R^*}) \subset B_{\alpha R^*}$. Indeed, if $\alpha \leq 1$ then it is obvious because from \eqref{eq:est_Tu} that $T(B_{\alpha R^*}) \subset T(B_{R^*}) \subset B_{\alpha R^*}$. Otherwise if $\alpha > 1$ let us write $B_{\alpha R^*} = B_{\frac{R_0}{2}} \cup S_1 \cup S_2$ where $S_1:= B_{R^*} \setminus B_{\frac{R_0}{2}}$, $S_2:= B_{\alpha R^*} \setminus B_{R^*}$. We already know that $T(B_{\frac{R_0}{2}}) \subset B_{R_1} \subset B_{R^*}\subset B_{\alpha R^*}$. Now from \eqref{eq:est_Tu} we find that $T(S_1) \subset B_{\alpha R^*}$, and from the definition of $R_0$ and again from \eqref{eq:est_Tu} we find that $T(S_2) \subset B_{\frac{\alpha}{2}R^*} \subset B_{\alpha R^*}$. \\

\end{itemize}

Finally since $B_{\alpha R^*}$ is a closed and convex subset of the Banach space $\mcl{X}(I)$, we can apply Schauder's fixed point theorem which yields the existence of a solution to \eqref{eq:non_local_quasilinear'}. 

\end{description}

\end{proof}

\subsection{Existence of strong solutions for relatively continuous operators}

We assume throughout this subsection that the map $\mathbb{A}: \mcl{X}(I) \lra RC(I)$ maps to the space of relatively continuous operators and that $\mathbb{A}(u)$ is globally bounded in $L^\infty(I;\mcl{L}(D,X))$. 

Let $\alpha \geq 1$, $\beta \geq 0$, $r > 1$ be fixed and assume that for each $u \in \mcl{X}(I)$, $\mathbb{A}(u)$ satisfies \ref{item:A1}- \ref{item:A2}. Denote by $\mathbb{A}_0: u \mapsto \mcl{A}_0(u)$ the operator given by the decomposition \ref{item:A1}\ref{item:b} and the subsequent $u \mapsto K_0(u), M_0(u)$ given in \eqref{eq:def_KA_MA}.\\

Naturally this gives rise to the following maps: 

\begin{enumerate}[label = (\roman*)]
\item $m_\delta: \mcl{X}(I) \lra (0,1]$
\item $m_\eta: \mcl{X}(I) \lra \R_+$
\item $\Gamma: \left\{\begin{array}{ccl}
\mcl{X}(I) &\lra& \R_+\\
u &\longmapsto& \norm*{(K_{\mathbb{A}(u)} + 1)^\alpha(M_{\mathbb{A}(u)} + 1)^\alpha}_{L^{\alpha r}(I)}^{\alpha}
\end{array}\right.$
\end{enumerate}
 
\begin{enumerate}[label = (\textbf{E}$_I'$)]
	\item \label{item:E_Iprime} For any $R \geq 0$, define the following quantities:
	\begin{equation*}
		\gamma_{\log}(R):= \sup_{\norm*{u}_{\mcl{X}(I)} = R} \cbk*{\pr*{\frac{\Gamma(u)}{m_\delta(u)}}^{r^*}\pr*{ 1 + m_\eta(u)m_\delta(u)^{s_2}\Gamma(u)^{s_1} + \log\pr*{\frac{\Gamma(u)}{m_\delta(u)}}} + m_\eta(u)}
	\end{equation*}
	\begin{equation*}
		\kappa_{\log}(R):= \sup_{\norm*{u}_{\mcl{X}(I)} = R} \cbk*{\log\pr*{\norm*{\mathbb{A}(u)}_\infty + \norm*{\mathbb{F}(u)}_{L^p(I;X)} + 1}}
	\end{equation*}
	and assume that there exists $h > 0$, such that for any $L > 0$, there exists $R_0 \geq 0$ such that for all $R \geq R_0$ we have
	\begin{equation}
	\label{eq:growth_eta_delta_Gamma2}
	\sup_{R_1, R_2 \in [R-L, R + L]}
	\frac{\gamma_{\log}(R_1) + \kappa_{\log}(R_2)}{\log\pr*{R}} \leq 1-h
	\end{equation}
	where $s_1(\alpha, \beta, r):=r^*\pr*{\frac{\beta + 1}{\alpha}-1}$ and $s_2(\alpha, \beta, r):= \frac{\beta + 1}{\alpha + 1} - \pr*{\frac{\alpha + \frac{1}{r}}{\alpha +1}}s_1(\alpha, \beta, r)$. 
	\end{enumerate}

\begin{thm}
	Assume that for each $u \in \mcl{X}(I)$, $\mathbb{A}(u)$ satisfies \ref{item:A1} and \ref{item:A2}. Moreover assume that $\mathbb{A}, \mathbb{F}$ satisfy \ref{item:E_Iprime} and \ref{item:E_II}. Then for any $x \in \Tr^p$, there exists a unique solution $u$ to \eqref{eq:non_local_quasilinear'}. 
\end{thm}

\begin{proof}
The growth condition \eqref{eq:growth_eta_delta_Gamma2} straightforwardly implies  the growth condition \ref{item:E_I}. \ref{item:E_II} is fulfilled by assumption. 
We simply apply \autoref{thm:existence_quasilinear'}. 
\end{proof}

\section{Conclusion and perspectives} 

This work gives new estimates of the maximal regularity constants of nonautonomous relatively continuous operators and describes the weak topologies of operators in $L^\infty(I;\mcl{L}(D,X))$ for which the maximal regularity is persistent.

These inquiries are used to derive a new well-posedness criterion for a general class of Cauchy problems for nonlocal in time and quasilinear equations. This criterion allows to bypass the standard Lipschitz continuity of the nonlinearities and instead requires growth conditions on the pointwise constant of maximal regularity.

We hope that these inquiries will help the study of nonlinear parabolic equations arising in various models involving time-delays or memory effects with moderately low-regularity coefficients in time. A first inquiry pursued by the authors is to find concrete applications of the result. 

Remaining inquiries would involve the existence and regularity of the semi-flow associated to \eqref{eq:non_local_quasilinear} and better asymptotic estimates of the maximal regularity constant to possibly improve the criterion \eqref{eq:growth_eta_delta_Gamma}.

\newpage

\bibliographystyle{abbrv}
\bibliography{master}

\end{document}